\documentclass[11pt]{amsart} 
\usepackage{amsmath, amsfonts}
\usepackage{amsthm}
\usepackage{amssymb}
\usepackage[centering, vcentering, marginratio=1:1, vscale=0.8, hscale=0.7, letterpaper]{geometry}
\usepackage{url} 
\usepackage[colorlinks,  citecolor=black, linkcolor = black, urlcolor = black]{hyperref} 
\usepackage{mathtools}
\usepackage{tikz-cd}
\usepackage{stmaryrd}
\usepackage{mathrsfs}
\usepackage{derivative}
\usepackage{enumerate}

\theoremstyle{plain}
\newtheorem{thm}{Theorem}[section]
\newtheorem{prop}[thm]{Proposition}
\newtheorem{cor}[thm]{Corollary}
\newtheorem{lem}[thm]{Lemma}
\newtheorem{question}[thm]{Question}

\theoremstyle{definition}
\newtheorem{defn}[thm]{Definition}

\newtheorem{defn-prop}[thm]{Definition-Proposition}

\theoremstyle{remark}
\newtheorem{remark}[thm]{Remark}

\DeclareMathOperator{\Hom}{Hom}

\DeclareMathOperator{\Spec}{Spec}

\DeclareMathOperator{\Proj}{Proj}

\DeclareMathOperator{\an}{an}
\DeclareMathOperator{\ma}{MA}

\DeclareMathOperator{\vol}{vol}
\DeclareMathOperator{\na}{NA}
\DeclareMathOperator{\trop}{trop}

\DeclareMathOperator{\Val}{Val}
\DeclareMathOperator{\ord}{ord}

\DeclareMathOperator{\qm}{qm}

\DeclareMathOperator{\fut}{Fut}

\DeclareMathOperator{\fs}{\mathrm{FS}}
\DeclareMathOperator{\cpsh}{\mathrm{CPSH}}

\DeclareMathOperator{\supp}{\mathrm{supp}}
\DeclareMathOperator{\Coeff}{\mathrm{Coeff}}
\newcommand{\tc}{\mathrm{tc}}

\renewcommand{\AA}{\mathbb{A}}

\newcommand{\CC}{\mathbb{C}}

\newcommand{\GG}{\mathbb{G}}

\newcommand{\NN}{\mathbb{N}}

\newcommand{\PP}{\mathbb{P}}
\newcommand{\QQ}{\mathbb{Q}}
\newcommand{\RR}{\mathbb{R}}

\newcommand{\TT}{\mathbb{T}}

\newcommand{\ZZ}{\mathbb{Z}}

\newcommand{\kk}{\Bbbk}

\newcommand{\cB}{\mathcal{B}}

\newcommand{\cF}{\mathcal{F}}

\newcommand{\cH}{\mathcal{H}}
\newcommand{\cI}{\mathcal{I}}

\newcommand{\cL}{\mathcal{L}}
\newcommand{\cM}{\mathcal{M}}

\newcommand{\cO}{\mathcal{O}}

\newcommand{\cR}{\mathcal{R}}

\newcommand{\cV}{\mathcal{V}}

\newcommand{\cX}{\mathcal{X}}

\newcommand{\fm}{\mathfrak{m}}

\newcommand{\ft}{\mathfrak{t}}

\newcommand{\la}{\langle}
\newcommand{\ra}{\rangle}

\newcommand{\bP}{\mathbb{P}}
\newcommand{\bC}{\mathbb{C}}
\newcommand{\bR}{\mathbb{R}}
\newcommand{\bA}{\mathbb{A}}
\newcommand{\bQ}{\mathbb{Q}}
\newcommand{\bZ}{\mathbb{Z}}
\newcommand{\bT}{\mathbb{T}}
\newcommand{\bN}{\mathbb{N}}

\newcommand{\bG}{\mathbb{G}}

\newcommand{\oX}{\overline{X}}
\newcommand{\oB}{\overline{B}}
\newcommand{\oGamma}{\overline{\Gamma}}

\allowdisplaybreaks

\title{K-semistability of Log Fano Cone Singularities}
\author{Yuchen Liu and Yueqiao Wu}
\address{Department of Mathematics, Northwestern University, Evanston, IL 60208, USA}
\email{yuchenl@northwestern.edu}

\address{Department of Mathematics, Johns Hopkins University, Baltimore, MD 21218, USA}
\email{ywu347@jhu.edu}

\begin{document}
\begin{abstract}
    We give a non-Archimedean characterization of K-semistability of log Fano cone singularities, and show that it agrees with the definition originally defined by Collins--Sz\'ekelyhidi. As an application, we show that to test K-semistability, it suffices to test special test configurations. We also show that special test configurations give rise to lc places of torus equivariant bounded complements.
\end{abstract}

\maketitle
\tableofcontents
\addtocontents{toc}{\protect\setcounter{tocdepth}{1}}

\section{Introduction}
K-stability of log Fano pairs has been extensively studied in the past decade. From the differential geometry point of view, the Yau--Tian--Donaldson conjecture for log Fano pairs is now a theorem stating that a log Fano pair admits a K\"ahler-Einstein metric if and only if it is K-polystable, thanks to the work of~\cite{BBJ, LTW, Li21, LXZ}. Algebro-geometrically, K-stability also turns out to be the right notion for forming moduli space of Fano varieties~\cite{Jia20, CP21, LWX, BX, ABHLX, Xu, BLX,  XZCM, XuZhuang20, BHLLX, LXZ}.

A local analogue of the above is K-stability of \emph{log Fano cone singularities}, also commonly known as \emph{local K-stability}. While we will recall the relevant definitions later, this was first studied in~\cite{MSY} and roughly speaking these singularities are generalizations of affine cones over log Fano varieties by allowing irrational polarizations. In a similar manner, the motivation for studying local K-stability is twofold. First, by extending the global K-stability theory to the local setting, Collins--Sz\'ekelyhidi~\cite{CS19} and C. Li~\cite{Li21} showed a YTD type result for log Fano cone singularities, stating that a normalized log Fano cone singularity admits a weak Ricci-flat K\"ahler cone metric if and only if it is K-polystable with respect to $\QQ$-Gorenstein test configurations. 
Second, from an algebro-geometric perspective, the recently resolved stable degeneration conjecture (see~\cite{XZ22, LWX}) describes a 2-step degeneration theory for degenerating klt singularities to K-semi/polystable log Fano cone singularities. This also in turn motivates the study of moduli theory of klt singularties, and there has been recent works toward this direction, see e.g. \cite{XZbdd, OdakaModuli, Che24}. 

While there are a lot of activities surrounding the study of local K-stability in recent years, most of the algebraic theory relies on testing local K-stability on refined classes of so-called \emph{special test configurations} and \emph{$\QQ$-Gorenstein test configurations}, as opposed to the more general class of \emph{normal test configurations} in the original definition, see~\cite{CSIrregular}. Thus, it is natural to ask if it is enough to test local K-stability using just special test configurations. Indeed, this is the case for the global theory, as first shown in~\cite{LX14} using the Minimal Model Program (MMP). In this paper, we give an affirmative anwer to the above question for local K-semistability, and our approach relies on a non-Archimedean (NA) characterization of local K-stability.

Let $(X, B; \xi)$ be a log Fano cone singularity, i.e., $(X, B)$ is a $\QQ$-Gorenstein klt pair with a good torus action fixing a cone point on $X$, and $\xi$ is a Reeb vector. A key feature of such objects is that $X$ can be thought of as an affine cone over a log Fano pair with respect to some multiple of the anti-canonical polarization whenever $\xi$ is a rational vector. In particular, there is a one-to-one correspondence between test configurations of $(X, B; \xi)$ and that of the quotient log Fano pair when $\xi$ is rational. The idea of~\cite{CSIrregular} is to define a Futaki invariant for normal test configurations so that it varies continuously with respect to $\xi$, and recovers the Futaki invariant for the corresponding log Fano pair when $\xi$ is rational. Following the same strategy, our guiding principle is to define non-Archimedean functionals depending continuously on the Reeb fields, thereby extending the global versions continuously to irrational polarizations. 
More precisely, to each normal test configuration $(\cX, \cB; \xi, \eta)$, as studied in~\cite{wu}, there is a naturally associated FS function $\varphi=\fs(\xi, \eta): X^{\an}\setminus\{o\}\to \RR$ on the Berkvovich analytification of $X$, and one can define the non-Archimedean Monge--Amp\`ere operator on the space of FS functions $\cH^{\na}(\xi)$. The operator defines probability measures of the form $\sum_v c_v\delta_v$, where $c_v>0$ should be thought of as a local intersection number, and $\delta_v$ is the dirac mass supported at some quasi-monomial valuation. Using mixed Monge--Amp\`ere measures as a main source of input, we can define non-Archimedean local Mabuchi $M^{\na}$ and Ding $D^{\na}$ functionals in a similar way as its global counterparts defined in~\cite{BHJ17}. We refer the readers to Section~\ref{section: NAfunctional} for precise definitions of these functionals. Our first result is the following continuity property.
\begin{thm}\label{thm-main-cts}
    For a fixed normal test configuration $(\cX, \cB; \xi, \eta)$ of a log Fano cone singularity $(X, B; \xi)$, the function 
    \[\xi\mapsto F^{\na}(\fs(\xi, \eta))\]
    is continuous on the Reeb cone for any $F\in \{M, D\}$. Moreover, when $\xi$ is rational, they recover the corresponding global NA functionals.
\end{thm}

The above theorem is a synopsis of the following more precise statements: Proposition~\ref{chap5-prop-local-global-mabuchi}, Theorem~\ref{chap5-thm-mabuchi-cts}, Proposition~\ref{chap5-prop-local-global-ding} and Corollary~\ref{chap5-cor-ding-cts}. This can be viewed as an analogous result in~\cite{CSIrregular} where they extended the Futaki invariant using the index character. In fact, in the case when test configurations are $\QQ$-Gorenstein, the NA Ding functional also recovers a NA Ding functional defined in successively greater generality in~\cite{LX18, LWX,Li21} via a limit slope computation. 

Apart from the variational formulation, there is also a local valuative criterion for K-semistability introduced in~\cite{XuZhuang20, HuangThesis}. Combining with the volume minimizing properties, it is shown in~\cite{HuangThesis} that K-semistability tested on special test configurations is equivalent to certain local $\delta$-invariant being at least 1. Below is a more precise version of the Li--Xu type theorem that we prove in the local case. This is also a generalization of~\cite[Theorem 2.9]{Li19}.
\begin{thm}\label{main-thm-stability-equivalence}
    Let $(X, B; \xi)$ be a log Fano cone singularity. Then the following are equivalent:
    \begin{enumerate}
        \item $(X, B; \xi)$ is K-semistable;
        \item $M^{\na}\geq 0$ on $\cH^{\na}(\xi)$;
        \item $D^{\na}\geq 0$ on $\cH^{\na}(\xi)$;
        \item $(X, B; \xi)$ is K-semistable with respect to special test configurations;
        \item $\delta(X, B;\xi)\geq 1$.
    \end{enumerate}
\end{thm}
Our approach to prove this theorem uses instead a non-Archimedean thermodynamical formalism developed in~\cite{BJ18, BBJ, BJNA1}, and relies on some non-Archimedean pluripotential theory developed in~\cite{wu,NAPP}. 
While we believe the same result should be true with K-polystability in place of K-semistability, 
let us remark briefly that the main difficulty seems to be understanding a local reduced uniform stability, in which case certain NA norms would not have nice continuity properties with respect to the Reeb fields.

As suggested in the global setting from \cite[Appendix A]{BLX}, it is often useful to understand special test configurations as coming from valuations which are lc places of bounded complements. Building on a $\bT$-equivariant version of~\cite[Theorem A.2]{BLX} (see Theorem \ref{thm:torus-comp}), we also prove a local analogue for log Fano cone singularities.

\begin{thm}\label{main-thm-complements}
    Let $n$ be a positive integer and $I\subset [0,1]\cap \QQ$ be a finite set. Then there exists a positive integer $N$ depending only on $n$ and $I$ such that for any $n$-dimensional log Fano cone singularity $(X, B, \TT=\GG_m^r; \xi)$ with $\Coeff(B)\subset I$ and $E$ a $\bT$-equivariant Koll\'ar component giving rise to a nontrivial special test configuration of $(X, B; \xi)$, there exists a $\TT$-equivariant $N$-complement $\Gamma$ of $(X, B, o)$ such that $E$ is an lc place of $(X, \Gamma)$.
\end{thm}

\subsection*{Organization}
In Section~\ref{section: prelim} we recall basic definitions of log Fano cone singularities and non-Archimedean pluripotential theory, where we also prove a couple auxiliary lemmas we need for NA Monge--Amp\`ere measures. In Section~\ref{section: localstab} and Section~\ref{section: localval} we carefully review the theory of K-stability of log Fano cones, and spell out in detail how it is compared to the K-stabilty theory for log Fano pairs. Along the way, we prove Theorem~\ref{main-thm-complements}. We introduce the relevant NA functionals defining local K-stability and prove Theorem~\ref{thm-main-cts} in Section~\ref{section: NAfunctional}. Along the way, we also give an algebraic formula for the $L^{\na}$ functional. Finally we prove Theorem~\ref{main-thm-stability-equivalence} in Section~\ref{section: pfmain}.

\subsection*{Acknowledgements}
The first author thanks   Harold Blum, Chenyang Xu,  Chuyu Zhou, and Ziquan Zhuang  for helpful discussions.
The second author thanks Mattias Jonsson for helpful discussions.  YL is partially supported by NSF CAREER grant DMS-2237139 and the Alfred P. Sloan Foundation. YW is partially supported by NSF grants DMS-1900025, DMS-2154380 and DMS-1926686.

\section{Preliminaries}\label{section: prelim}
\subsection{Valuations and norms}
\subsubsection{The Berkovich analytification}
Let $X=\Spec R$ be a normal affine variety over an algebraically closed field $\kk$ of characteristic zero. The \emph{Berkovich analytification} of $X$ with respect to the trivial absolute value on $\kk$ is the set $X^{\an}$ consisting of all \emph{semivaluations} $v: \kk(X)\to \RR\cup \{+\infty\}$ extending the trivial valuation on $\kk$. Here, by a semivaluation, we mean a real valued valuation $v: \kk(Y)^\times\to \RR$ on a subvariety $Y\subseteq X$. As a topological space, $X^{\an}$ is equipped with the weakest topology such that $v\mapsto v(f)$ is continuous for all $f\in R$. A valuation $v$ is \emph{divisorial} if there is a prime divisor $E$ over $X$ such that $v=c\ord_E$ for some $c>0$; it is \emph{quasi-monomial} if there is a log smooth model $(Y, D=\sum_{i\in I} D_i)$ over $X$ such that around the generic point $\eta$ of an irreducible component of $\bigcap_{i\in J} D_i$ with $J\subseteq I$ and an algebraic coordinate system $(z_i)$ at $\eta$, there are $\alpha_i\geq 0$ such that for any $f\in R$, if $f=\sum_\beta c_\beta z^\beta\in \widehat{\cO_{Y, \eta}}$ with $c_\beta$ either zero or a unit, then 
\[
v(f)=\min \{\sum_i\alpha_i\beta_i\mid c_\beta\neq 0\}.
\]
\subsubsection{$\TT$-invariant valuations}
Now suppose $X$ has a good algebraic torus $\TT=\GG_m^r$ action. Recall that a $\bT$-action is good if it is effective 
and there exists a $\bT$-fixed closed point $o\in X$ (called the cone point) that is contained in the closure of every $\bT$-orbit. We write $M\coloneqq \Hom(\TT, \GG_m)$ for the weight lattice, and $N\coloneqq M^\vee = \Hom(\GG_m, \TT)$ the dual lattice. Then we have a weight decomposition $R=\bigoplus_{\alpha}R_\alpha$. Put $N_\RR = N\otimes_\ZZ \RR$, and $\Lambda\coloneqq \{\alpha: R_\alpha \neq 0\}$. A semivaluation on $v\in X^{\an}$ is \emph{$\TT$-invariant} if $v(\sum_\alpha f_\alpha) = \min_\alpha v(f_\alpha)$. Given any vector $\xi\in N_\RR$, the torus action induces a natural $*$-operation (see~\cite[\S 5.2]{Berk12}) which takes a $\TT$-invariant valuation $v$ to a $\TT$-invariant valuation $\xi*v$ via
\[
(\xi*v)(f_\alpha)\coloneqq v(f_\alpha)+\la\xi, \alpha\ra.
\]
\subsubsection{Filtrations and norms}
Let $R=\bigoplus_\alpha R_\alpha$ be as above. 
\begin{defn}
    A \emph{$\TT$-invariant} filtration on $R$ is a family of vector subspaces $\cF^\lambda R=\bigoplus_\alpha \cF^\lambda R_\alpha$ of $R$ for $\lambda\in \bR$ such that 
    \begin{enumerate}
        \item (decreasing) $\cF^\lambda R\subseteq \cF^{\lambda'}R$ if $\lambda\geq \lambda'$;
        \item (left-continuous) $\cF^\lambda R=\bigcap_{\lambda'<\lambda}\cF^{\lambda'}R$;
        \item (multiplicative) $\cF^\lambda R_\alpha\cdot \cF^{\lambda'}R_{\alpha'}\subseteq \cF^{\lambda+\lambda'}R_{\alpha+\alpha'}$.
        \item (exhaustive)  $\cF^\lambda R_\alpha=R_\alpha$ if $\lambda \ll 0$ and $\cF^\lambda R_\alpha=0$ if $\lambda\gg 0$.
    \end{enumerate}
\end{defn}
As observed in~\cite{Nystrom, BHJ17}, $\TT$-invariant filtrations on $R$ are in one-to-one correspondence with $\TT$-invariant norms on $R$ via
\[
\chi(f_\alpha)\coloneqq \sup\{\lambda\mid f_\alpha\in \cF^\lambda R_\alpha\},
\]
for any $f_\alpha\in R_\alpha$, and conversely given such a norm $\chi$, the filtration is defined by
\[
\cF^\lambda R_\alpha\coloneqq\{f_\alpha\in R_\alpha\mid \chi(f_\alpha)\geq \lambda\}.
\]
In particular, any $\TT$-invariant valuation defines a $\TT$-invariant filtration. We will be using filtrations and norms interchangeably.
\subsection{Log Fano cone singularities}
We fix a normal affine variety $X=\Spec R$ over $\kk$ with a good $\TT=\GG_m^r$ action,
and denote by $o$ the cone point, with $\fm$ being the maximal ideal defining the cone point.
\begin{defn}
    The \emph{Reeb cone} is 
    \[\ft_\RR^+\coloneqq \{\xi\in N_\RR: \la \xi, \alpha\ra >0, \ \forall \alpha\in \Lambda\setminus\{0\}\}.\]
    A vector $\xi\in \ft_\RR^+$ is called a \emph{Reeb field}.
    A \emph{log Fano cone singularity} is the data $(X, B, \TT; \xi)$ such that 
    $X$ is a normal affine $\TT$-variety with a good torus action, $B$ is a $\TT$-invariant vertical $\QQ$-divisor on $X$ such that $K_{(X,B)}\coloneqq K_X+B$ is $\QQ$-Cartier, $(X, B)$ has klt singularities, and $\xi$ is a Reeb field. 
\end{defn}
\begin{defn}
    A Reeb field $\xi$ is called \emph{quasi-regular} or \emph{rational} if $\xi\in \ft_\QQ^+\coloneqq \ft_\RR^+\cap N_\QQ$. It is called \emph{irregular} or \emph{irrational} otherwise. For a quasi-regular Reeb field $\xi$, we define its \emph{primitive vector} $\hat \xi\coloneqq l\xi$ as the multiple of $\xi$ such that $\hat\xi\in N\cap \ft_\RR^+$ is a primitive integer vector, i.e., $l$ is smallest with $\la l\xi, \alpha\ra\in \ZZ, \forall \alpha\in \Lambda$. 
\end{defn}
\begin{defn}
    The \emph{non-Archimedean (NA) link} of $X$ at $o$ is 
    \[X_0\coloneqq \{v\in X^{\an} : v(\fm)=0\}.\]
    We will denote by $X_0^{\TT}$ the set of $\TT$-invariant points in $X_0$.
\end{defn}
\begin{remark}
    We remark that the NA link usually refers to the space $\Val_{X,x}$ in the literature, see e.g.~\cite{Fantini}.
    As explained in~\cite{wu}, the NA link we defined is the set $\{r^{\na}=1\}$ for certain NA reference cone metric $r$ on the Berkovich analytification $X^{\an}$, and hence a NA analog of the \emph{link} in the Sasakian geometry literature, see~\cite{sasaki}.
\end{remark}

Given a log Fano cone singularity, the Reeb vector $\xi$ defines a $\TT$-invariant quasi-monomial valuation $v_\xi$ via $v_\xi(f_\alpha)=\la \xi, \alpha\ra$ for any $f_\alpha\in R_\alpha$. We will write $A(\xi)\coloneqq A_{(X, B)}(v_\xi)$. The following is well-known.
\begin{prop}\label{prelim-prop-quotient-log-fano}
    Let $(X=\Spec R, B; \xi)$ be a log Fano cone singularity with $\xi\in \ft_\QQ^+$, then the quotient of $X\setminus\{o\}$ by $\la\xi\ra$ is a log Fano pair.
\end{prop}
\begin{proof}
   After rescaling we may assume that $\xi\in N$ is a primitive vector. Since $\xi$ generates a faithful $\GG_m$-action, one can think of $X\setminus\{o\}$ as the complement of the zero section in the total space of an ample orbiline bundle over some orbifold (see~\cite{RossThomas}). In algebro-geometric terms, $X\setminus \{o\}$ is a Seifert $\GG_m$-bundle over a projective variety $V = \Proj R$. It is shown in~\cite{kollar} that the coordinate ring $R$ with the new grading induced by the $\xi$ action can be viewed as the section ring of some ample $\bQ$-Cartier  $\QQ$-divisor $L$ on $V$:
    \[R = \bigoplus_\alpha R_\alpha = \bigoplus_{k\in \bZ_{\geq 0}}\left(\bigoplus_{\la \xi, \alpha\ra = k}R_\alpha\right) = \bigoplus_{k\in \ZZ_{\geq 0}} H^0(V, \cO_V(\lfloor kL\rfloor)).\]
    If the fractional part of $L$ is given by $\sum_i \frac{a_i}{b_i}D_i$ where $a_i<b_i$ are positive integers and $\gcd(a_i,b_i)=1$, we put $L_{\mathrm{orb}}\coloneqq \sum_i(1-\frac{1}{b_i})D_i$ as the orbifold boundary of $V$ to remember the orbifold structure. 
    Let $B_V=L_{\mathrm{orb}}+B'$, where $B'$ is the $\bQ$-divisor whose pullback under $\pi: X\setminus\{o\}\to V$ is $B$. Then $(V, B_V)$ is the quotient by $\la\xi\ra$.
    To see that $(V, B_V)$ is log Fano, this is the content of~\cite[Proposition 3.14]{kollarMMP} when $L_{\rm orb}=0$. We include a proof of the general case for the sake of completeness. 
    Note that by~\cite[Proposition 16]{kollar} we have $\mathrm{Pic}(X)$ is trivial, and that $\mathrm{Cl}(X) = \mathrm{Cl}(V)/\la L\ra$, we have $rL\sim mK_{(V,B_V)}$ for some $r\in \QQ$ since $mK_{(X,B)}$ is Cartier for $m$ divisible enough. Let $\tilde{X}$ be the total space of of the line bundle, obtained by a weighted blowup at $o$. Then the exceptional disivor $E\subset \tilde{X}\xrightarrow{p} X$ is isomorphic to $V$. By adjunction, we have
    \[(K_{\tilde{X}}+ B_{\tilde{X}}+E)|_E = K_V+B_V = A_{(X,B)}(E)E|_E  = A_{(X,B)}(E)(-L)|_E. \]
    Hence $r= -A_{(X,B)}(E)<0$ since $(X,B)$ has klt singularities, and so $(V,B_V)$ is a log Fano pair.
\end{proof}

Using this, we can relate the log discrepancy function on $X$ with the one on the quotient log Fano, and its restriction to $X_0^{\TT}$.
\begin{lem}\label{LogDiscrepancy}
Let $(X,B) = C(V,B_V,L=-rK_{(V, B_V)})$ be the cone over a log Fano pair. Then
\[A_{(X,B)}(v)= A_{(V, B_V)}(v),\]
for all quasi-monomial valuations $v\in X_0^{\TT}$.
\end{lem}

\begin{proof}
This is already done in multiple places, see e.g.~\cite{PS, Moraga}.
    Let's first take care of the case when $v$ is given by a prime divisor. Assume
\[F\subseteq V' \xrightarrow[]{
\mu} V\]
is a prime divisor over $V$, with $V'$ normal and $v=\ord_F$.

We have the following commutative diagram:
\begin{eqnarray*}\label{commdiag}
\begin{tikzcd}
E \arrow[dd, "\pi'"] \arrow[r, phantom, sloped, "\subset"]  &\tilde{X}' \arrow[dd, "\pi'"] \arrow[rr, "\tilde{\mu}"] &  & \tilde{X} \arrow[dd, "\pi"] \arrow[r, "r"] & X \\
 & &  &    &   \\
F \arrow[r, phantom, sloped, "\subset"]   &V' \arrow[rr, "\mu"]                                    &  & V  &  
\end{tikzcd}
\end{eqnarray*}
where $\tilde{X}$ denotes the partial resolution of $X$ extracting the divisor isomorphic to $ V$,  $E = (\pi'^*F)_{\mathrm{red}}$.

We now follow the proof of~\cite[Propostion 5.20]{KM98}. Let $\tilde{B}$ be such that $\pi^*K_{(V, B_V)}=K_{(\tilde{X}, \tilde{B})}$ and $\tilde{B}'$ the strict transform of $\tilde{B}$. Let $s$ be the Weil index of $F$ along $E$. Near the generic point $e\in E$, one has
\begin{align*}
    K_{(\tilde{X}',\tilde{B}')} + E &= \tilde{\mu}^*K_{(\tilde{X}, \tilde{B})}+ A_{(X,B)}(E) E\\
    &= \tilde{\mu}^*\pi^* K_{(V,B_V)} + A_{(X,B)}(E) E\\
    &= \pi'^*\mu^*K_{(V, B_V)} + A_{(X,B)}(E) E,
\end{align*}
and 
\begin{align*}
    K_{(\tilde{X}',\tilde{B}')} + E &= \pi'^*K_{(V', B_V')}+ sE\\
    &=\pi'^*\mu^*K_{(V, B_V)} + (A_{(V, B_V)}(F)-1)\pi'^*F + sE\\
    &=\pi'^*\mu^*K_{(V, B_V)} + sA_{(V,B_V)}(F) E.
\end{align*}
Thus $A_{(X,B)}(v) = A_{(X,B)}(s^{-1}\ord_E) = s^{-1} sA_{(V,B_V)}(F) = A_{(V,B_V)}(F). $ This proves the identity for all divisorial valuations. Now if $v$ is quasi-monomial, then in view of the fact that $X_0^\TT\subset V^{\an}$ (see~\cite[Lemma 5.11]{wu}), $A_{(X,B)}(v)$ can be written as a linear combination of log discrepancies of divisorial valuations on some log smooth model, and hence the equality extends to quasi-monomial valuations.
\end{proof}

\begin{lem}\label{logdiscreptranslation}
Let $v\in X_0^\TT$ be a quasi-monomial valuation. Then we have
\[A_{(X,B)}(\xi*v) = A_{(X,B)}(\xi)+A_{(X,B)}(v),\]
for all $\xi$ in the Reeb cone of $X$.
\end{lem}
\begin{proof}
    We will write $A$ for $A_{(X, B)}$. If $\xi$ is rational, then by~\cite[Lemma 5.11]{wu}, $v$ can be seen as a valuation on the quotient $(X\setminus\{o\})/\langle \xi\rangle$, and $\xi*v$ is a weight $(1,1)$-blowup (see e.g.~\cite[Proof of Proposition 4.25]{Xusurvey}), in which case the equality follows. 
    In general, approximate $\xi$ by rational $\xi_i$'s. Then by lower-semicontinuity of the log discrepancy \cite{JM12, bdffu}, one has
    \[A(\xi*v)\leq \liminf_i (A(\xi_i)+A(v)) = A(\xi)+ A(v).\]
    On the other hand, we may pick $\eta\in \ft_\RR^+$ such that $\xi+\eta$ is rational. Then 
    \[A(\xi)+A(\eta)+A(v) = A(\xi+\eta)+A(v) = A((\xi+\eta)*v) \leq A(\eta) + A(\xi*v).\]
    Combining the above two inequalities we get what we want. 
\end{proof}

\subsection{Non-Archimedean pluripotential theory}
In the remaining part of this section, we recall some non-Archimedean pluripotential theory, and
study how the support of the NA Monge--Amp\`ere measures change with respect to different Reeb fields $\xi$. Throughout, we will work with $X^{\an}$, the Berkovich analytification of $X$ with respect to the trivial valuation on $\kk$. Using the theory of forms and currents on Berkovich spaces developed in~\cite{CLD12}, we studied psh functions coming from test configurations and in particular the Monge--Amp\`ere operator in~\cite{wu, NAPP}. More precisely, to each $\TT$-invariant norm $\chi$ of finite type, i.e. $\mathrm{gr}_\chi R$ is finitely generated, one can associate a Fubini-Study (FS) function of the form 
\[\fs(\xi, \chi)\coloneqq \max_{f_\alpha\in A}\left\{\frac{\log|f_\alpha|+\chi(f_\alpha)}{\la\xi,\alpha\ra}\right\}\]
on $X^{\an}\setminus\{o\}$, where $A$ is a finite set determined by $\chi$ (see~\cite[Lemma 5.3]{wu}), 
which is psh-approachable in the sense of Chambert-Loir and Ducros, and $\log|f_\alpha|: X^{\an}\to \RR\cup \{\infty\}, v\mapsto -v(f_\alpha)$. In particular, when $\chi=0$ is the trivial norm, we get the trivial potential representing the reference NA metric on $X^{\an}$, which we denoted by $\varphi_\xi$.
This allows us to define a Monge--Amp\`ere operator $\ma: \cH(\xi)\to \cM(X_0)$ from the space of FS functions to the space of Radon probability measures on $X$ via
\[\ma(\varphi)\coloneqq \frac{1}{\vol(\xi)}(d'd''\varphi)^{n-1}\wedge d'd''\max\{\varphi_\xi, 0\}.\]
It is further shown in~\cite{wu, NAPP}, following the strategy developed in~\cite{BJ18}, that this extends continuously to an operator $\ma: \cpsh(\xi)\to \cM^1(\xi)$, where $\cpsh(\xi)\supset \cH(\xi)$ is a class of continuous $\xi$-psh functions, and $\cM^1(\xi)$ denotes the space of Radon probability measures of finite energy supported on $X_0$.

Similar to the complex analytic theory, \emph{currents} (and in particular measures) are defined as linear functionals on the space of compactly supported differential forms, and they satisfy similar properties as their complex counterpart. For example, here is a fact we need later:
\begin{prop}\label{prop-current-supp}
    Let $\varphi: W\to Z$ be a compact morphism of Berkovich analytic spaces, and $T$ a current on $W$. Then $\varphi_*T$ is a current on $Z$ such that $\supp(\varphi_*T)\subset \varphi(\supp(T))$.
\end{prop}

We now collect a couple facts that we need later concerning the support of Monge--Amp\`ere measures. From now on, we fix a norm $\chi$ of finite type, and denote by 
\[\varphi = \fs(\xi, \chi) := \max\left\{\frac{\log|f_j|+\chi(f_j)}{\langle \xi,\alpha_j\rangle}\right\}\in \cH(\xi)\]
the FS function corresponding to $\chi$, with the polarization given by $\xi$. 
Let $\xi_i$ be a sequence of Reeb fields converging to $\xi$. We write 
\[\varphi^i= \fs(\xi_i, \chi):=\max\left\{\frac{\log|f_j|+\chi(f_j)}{\langle\xi_i,\alpha_j\rangle}\right\}\in \cH(\xi_i)\]
for the FS functions corresponding to the same norm $\chi$, but with respect to Reeb fields $\xi_i$. In a similar manner, we set 
\[\mu^i_t:=\mu_t(\varphi^i, \xi_i):= \sum_{j=0}^{n-1} (d'd''\varphi^i)^j\wedge (d'd''\varphi_{\xi_i})^{n-1-j}\wedge \max\{\varphi_{\xi_i}, -t\},\]
and 
\[\mu_t:=\mu_t(\varphi, \xi):= \sum_{j=0}^{n-1} (d'd''\varphi)^j\wedge (d'd''\varphi_{\xi})^{n-1-j}\wedge \max\{\varphi_{\xi}, -t\}\]
for $t\in \RR_{\geq 0}.$

\begin{lem}\label{ctsMAsupp}
    For fixed $t\in \RR_{\geq 0}$, the measures $\mu_t^i$ converge weakly to $\mu_t$ as $i\to \infty$.
\end{lem}
\begin{proof}
    This is an immediate consequence of~\cite[Lemma 2.11]{NAPP}.
\end{proof}

\begin{lem}\label{MeasureTranslation}
    For fixed $\xi$, $\mu_t =(t\xi)_*\mu_0$, where 
    $t\xi : X^{\an}\to X^{\an}$ is given by $ v\mapsto (t\xi)*v$.
\end{lem}

\begin{proof}
    This is done by repeating the proof in~\cite[Corollary 5.7]{wu}. Indeed, the same proof shows $\mu_t$ is well defined for all $t\in \RR_{\geq 0}$, and is a finite sum of Dirac masses at $\TT$-invariant valuations. Restrict to a fine enough cell in a convenient cell decomposition and assume $0<t\ll 1$. Write $\varphi$ as pullback of smooth approximations $\varphi_\varepsilon$ using regularized max, see~\cite[Proof of Proposition 5.5]{wu}, and $\max\{\varphi_\xi, -t\}$ as the pullback of a function of the form $\max\{-\frac{x_j}{\langle \xi,\alpha_j\rangle}, -t\}$. A linear change of variable $l: x_j\mapsto x_j' = x_j-t\langle\xi,\alpha_j\rangle$ on the cell gives 
    $$\varphi +t = \lim_{\varepsilon\to 0} f_{\trop}^*(\varphi_\varepsilon\circ l) $$ and
    $$\max\{\varphi_\xi, -t\} = f_{\trop}^* \max\{-\frac{x_j'}{\langle \xi,\alpha_j\rangle}, 0\}.$$

    Note that $d'd''\varphi = \lim_{\varepsilon\to 0}f_{\trop}^*(d'd''\varphi_\varepsilon) = \lim_{\varepsilon\to 0}f_{\trop}^*(d'd''(\varphi_\varepsilon\circ l))$, it now remains to show that 
    \[d'd''\max\{-\frac{x_j'}{\langle \xi,\alpha_j\rangle}, 0\} = l_* d'd''\max\{-\frac{x_j}{\langle \xi,\alpha_j\rangle}, 0\}\]
    as a positive $(1,1)$-current on a cell. This is a consequence of~\cite{lagerberg}. Indeed, for any compactly supported smooth $(n-1, n-1)$-form $\omega$ on the cell, we have
    \begin{align*}
    \langle d'd''\max\{-\frac{x_j'}{\langle \xi,\alpha_j\rangle}, 0\}, \omega\rangle &= \int \max\{-\frac{x_j'}{\langle \xi,\alpha_j\rangle}, 0\} d'd''\omega  \\
    &= \int l^*(\max\{-\frac{x_j}{\langle \xi,\alpha_j\rangle}, 0\} )d'd''(l^*\omega) \\
    &= \langle l_*d'd''\max\{-\frac{x_j}{\langle \xi,\alpha_j\rangle}, 0\}, \omega\rangle.
    \end{align*}
\end{proof}

The following lemma shows that integrating the log discrepancy against these MA measures behaves well under taking limits, even though the log discrepancy function is not continuous. 
\begin{lem}\label{Aconvergence}
    Notation as before. For a fixed finite type norm $\chi$, we have 
    \[\int_{X^{\an}\setminus\{o\}} A(v) \mu^i_t\to \int_{X^{\an}\setminus\{o\}} A(v) \mu_t\]
    as $i\to \infty$.
\end{lem}
\begin{proof}
    Let $\{f_\alpha\}_{\alpha\in A}$ be a collection of $\TT$-invariant functions generating the norm $\chi$ and an $\fm$-primary ideal. Take a $\TT$-equivariant log smooth pair $\pi: (Y, D)\to X$ resolving the ideal
    $(f_\alpha, \alpha \in A)$ such that in local algebraic coordinates $\pi^*f_\alpha$ can be written as $\pi^*f_\alpha = u_\alpha\prod_{j=1}^r z_j^{\ell_j(\alpha)}$ for some unit $u_\alpha$.
    Write $W$ to be the set of valuations with center on $D$. 
    Note that $\pi$ induces a morphism of the corresponding analytic spaces $\pi^{\an}: Y^{\an}\to X^{\an} $ which is an isomorphism outside $(\pi^{an})^{-1}(o)$. Thus $(\pi^{\an})^*\varphi$ defines a continuous psh function on $Y^{\an}\setminus(\pi^{an})^{-1}(o)$, and thus we get well-defined measures
    $\nu_t^i:=\mu_t((\pi^{\an})^*\varphi, \xi_i)$ and
    $\nu_t:= \mu_t((\pi^{\an})^*\varphi, \xi)$ on
    $Y^{\an}\setminus(\pi^{an})^{-1}(0)$. 
    We will prove the lemma in three steps:
    \begin{enumerate}
        \item Step 1: The support of measures $\nu_t^i$ and $\nu_t$ vary within a compact subset of $\mathrm{QM}(Y,D)$.

        By construction, the measures $\nu_t^i, \nu_t$ have finite support in $W$, viewed as a subset of $Y^{\an}$. Thus on $W$, we can write the measures
        $\nu_t = \nu_t((\pi^{\an})^*\varphi, \xi)$ where  $(\pi^{\an})^*\varphi$ takes the form:
        \[(\pi^{\an})^*\varphi = \max\{\frac{\log|f_\alpha|+\chi(f_\alpha)}{\langle \xi,\alpha\rangle}\} = \max\{\frac{\sum_{j=1}^r\ell_j(\alpha)\log|z_j|+\chi(f_\alpha)}{\langle \xi,\alpha\rangle}\}\]
        Now using the coordinates $\{z_j\}$ as local moment maps on $W\subset Y^{\an}$, we have that each $x\in \mathrm{supp}(\nu_t)$ lies in the $n$-dimensional faces of the characteristic polyhedron $\Sigma_g^{(n)}$ for some moment map $g = (z_{j_1}, \cdots, z_{j_p}) $,  $1\leq j_1< \cdots j_p\leq r$, defined around $x$, see~\cite[\S (2.3.4)]{CLD12}. By definition, we have $\Sigma_g^{(n)}\subset \mathrm{QM}(Y,D)$.
        Note that the chosen model $(Y, D)$ only depends on the norm $\chi$, the same argument shows that the support of measures $\nu_t^i$ are also contained in $\mathrm{QM}(Y,D).$

        Next we claim that there is some compact subset $K\subset  \mathrm{QM}(Y, D)$ containing the support of all measures  $\nu_t^i, \nu_t$ for $i\gg 0$. Indeed, if $x$ is in the support of any of the measures $\nu_t^i$, then as shown in~\cite[Corollary 5.7]{wu}(see also~\cite[Proposition 5.7.4]{CLD12}), after tropicalizing, $x$ corresponds to the unique solution of a system of linear equations $A_ix=b_i$, where the entries in $A_i, b_i$ depend continuously on $\xi$. Thus $\|x\|\leq \|A_i^{-1}\|\|b_i\|$ is bounded for $i\gg 0$. We can then choose $K$ to contain the support of all measures.

        \item Step 2: $\pi^{\an}_*\nu_t = \mu_t$ and $\pi^{\an}_*\nu_t^i = \mu_t^i$.
        After approximating the $\max$ function by smooth modifications (see~\cite[Proposition 6.3.2]{CLD12},~\cite[Lemme 5.18]{Demailly}) and proceed as in~\cite{wu}, we can without loss of generality assume $\varphi$ is smooth, and $\mu_t, \nu_t$ are $(n,n)$-forms. Then we must show
        \[\pi^{\an}_*[(\pi^{\an})^*\mu_t] = [\mu_t], \]
        where in the above expression, $\mu_t$ denotes the $(n,n)$-form, and $[\mu_t]$ denotes the current given by the $(n,n)$-form. Indeed, for any compactly supported smooth function $u$ on $X^{\an}\setminus\{o\}$, we have 
        \begin{align*}
            \langle \pi^{\an}_*[(\pi^{\an})^*\mu_t], u\rangle = \int_{Y^{\an}\setminus (\pi^{\an})^{-1}(0)} (\pi^{\an})^*u(\pi^{\an})^*\mu_t = 
            \int_{X^{\an}\setminus\{0\}} u \mu_t.
        \end{align*}
        Hence we have the first assertion. The assertion for $\mu_t^i, \nu_t^i$ follows from the very same argument.
        
        \item Step 3: Combining the previous two steps, and Proposition~\ref{prop-current-supp}, we have that
        \[\supp(\mu_t), \supp(\mu_t^i) \subset K\subset \mathrm{QM}(Y,D).\]
        Note that the log discrepancy function is continous on $\mathrm{QM}(Y, D)$ by~\cite{JM12}. We therefore have 
        \begin{align*}
            \int_{X^{\an}\setminus \{0\}} A(v)\mu_t^i &= \int_{Y^{\an}\setminus (\pi^{\an})^{-1}(0)} A\circ \pi^{\an}(v)\nu_t^i \\
            &=\int_{K}A\circ \pi^{\an}(v)\nu_t^i  
            \to \int_{K}A\circ \pi^{\an}(v)\nu_t\\
            &= \int_{Y^{\an}\setminus (\pi^{\an})^{-1}(o)} A\circ \pi^{\an}(v)\nu_t\\
            &= \int_{X^{\an}\setminus \{o\}} A(v)\mu_t.
        \end{align*}
    \end{enumerate}
\end{proof}

\section{Local K-stability}\label{section: localstab}
We now provide a careful comparison between K-stability of log Fano pairs and K-stability of log Fano cone singularities. Most of the results in this section should be known to experts, and can be found in~\cite{MSY, CSIrregular,CS19, Livolmin, LiIzumi, LX18, LWX, Li21}.
\subsection{K-stability of log Fano pairs}
In this section, we largely follow~\cite{BHJ17, Xubook} to give a quick overview of K-stability of log Fano pairs. Throughout, let $(V, B_V=\sum_i c_i B_i)$ be a log Fano pair of dimension $n$, i.e. a projective klt pair with $L\coloneqq -(K_V+B_V)$ ample. Assume there is a torus $\TT$ acting on the log Fano pair.
\begin{defn}
    A \emph{($\TT$-equivariant ample) test configuration} for $(V, B_B; L)$ is a triple $(\cV, \cB_\cV; \cL)$ with $\cV$ normal, and a flat proper morphism $\pi: \cV \to \AA^1$ such that
    \begin{enumerate}
        \item there is a $\GG_m$-action on $\cV$ lifting the canonical action on $\AA^1$;
        \item $\TT$ acts on $\cV$ fiberwise which agrees with the $\TT$ action on $V$. Further, the action commutes with the $\GG_m$-action, and there is a $\TT\times \GG_m$-equivariant isomorphism $\phi: \cV\times_{\AA^1} (\AA^1\setminus\{0\})\cong V\times (\AA^1\setminus\{0\})$;
        \item $\cB_\cV = \sum_i c_i \overline{\GG_m\cdot \phi_1^{-1}(B_i)}$, where $\phi_1: \cV_1\cong V$ is the isomorphism over $1\in \AA^1$, and the closure is taken in $\cV$.
        \item $\cL$ is an ample $\GG_m$-linearized $\QQ$-line bundle such that $\phi$ extends to an isomorphism $\phi: (\cV_1, \cL_1)\cong (V, L)$.
    \end{enumerate}
\end{defn}
The above definition gives the most general class of normal test configurations in K-stability theory. 
There are other different classes of test configurations, and we list them below.
\begin{defn} \label{defn-global-tc}
A test configuration $(\cV, \cB_\cV; \cL)$ is 
\begin{enumerate}
    \item a \emph{$\QQ$-Gorenstein test configuration} if $\cL\sim_\QQ -K_{(\cV, \cB_\cV)}$;
    \item a \emph{weakly special test configuration} if $\cV_0$ is integral, $(\cV, \cV_0+\cB_\cV)$ is lc and $\cL\sim_\QQ -K_{(\cV, \cB_\cV)}$;
    \item a \emph{special test configuration} if $(\cV, \cV_0+\cB_\cV)$ is plt and $\cL\sim_\QQ -K_{(\cV, \cB_\cV)}$;
    \item a \emph{product test configuration} if there exists an $\GG_m$-equivariant isomorphism $(\cV, \cB_\cV) \cong (V, B_V)\times \AA^1$ over $\AA^1$ where the $\GG_m$-action on $\cV$ is given by a $\GG_m$-action on $V$ and the usual $\GG_m$-action on $\AA^1$.
\end{enumerate}
\end{defn}

\begin{remark}
    We remark that there are slightly different definitions for weakly special test configurations in the literature. For example, in~\cite{BLX}, non-integral central fibers are allowed in weakly special test configurations. We need an integral central fiber to draw comparison to the local case later. 
\end{remark}

Given a normal polarized scheme $(V, L)$ of dimension $n$ with a $\GG_m$-action, there is a weight decomposition on $H^0(V, mL)=\bigoplus_{\lambda\in \ZZ} H^0(V, mL)_\lambda$.
Define the total weight as $w_m(V, L)\coloneqq \sum_{\lambda\in \ZZ} \lambda h^0(V, mL)_\lambda$. We will also write $N_m(V, L)\coloneqq h^0(V, mL)$. Via the equivariant Riemann-Roch and asymptotic Riemann-Roch theorems, the above two invariants admit the following asymptotic expansions for $m\gg 0$:
\begin{align*}
    N_m(V, L) = a_0(V, L) m^n + a_1(V, L) m^{n-1} + O(m^{n-2});\\
    w_m(V, L) = b_0(V, L) m^{n+1} + b_1(V, L) m^n + O(m^{n-1}),
\end{align*}
where 
\[a_0(V, L) = \frac{(L^n)}{n!}, a_1 = -\frac{(K_V\cdot L^{n-1})}{2(n-1)!}.\]
Thus there is an asymptotic expansion:
\[\frac{w_m(V, L)}{mN_m(V, L)} = F_0(V, L) + F_1(V,L)m^{-1}+\cdots, \]
where $F_0 = \frac{b_0(V,L)}{a_0(V,L)}$, and $F_1 = \frac{a_0(V, L)b_1(V, L)-a_1(V, L)b_0(V, L)}{a_0(V, L)^2}$.

\begin{defn}\label{defn-global-fut}
    Let $(\cV, \cB_\cV = \sum_i c_i \cB_i; \cL)$ be a test configuration for $(V, B_V; L)$. Let $(\cV_0, \cB_0 = \sum_i c_i\cB_{i,0})$ be the central fiber, and write $\cL_0, \cL_{i,0}, $ for $\cL$ restricted to the central fiber and to the corresponding divisor $\cB_{i,0}$.
    The \emph{Futaki invariant} $\fut(\cV, \cB_\cV; \cL)$ of the test configuration is 
    \[\fut(\cV, \cB_\cV; \cL)\coloneqq -2F_1(\cV_0,\cL_0)+ \frac{\sum_ic_i(a_0(\cV_0, \cL_0)b_0(\cB_{i,0}, \cL_{i,0}) - a_0(\cB_{i,0}, \cL_{i,0})b_0(\cV_0, \cL_0))}{a_0(\cV_0, \cL_{0})^2}.\]
\end{defn}

\begin{thm}[{\cite[Proposition 2.17]{Xubook}, see also~\cite{Wang, Odaka, BHJ17}}]\label{thm-global-intersection}
    The Futaki invariant of a normal test configuration has the following intersection-theoretic formula:
    \[\fut(\cV, \cB_\cV; \cL) = \frac{(K_{(\bar\cV, \bar \cB_\cV)/\PP^1}\cdot \bar \cL^n)}{(L^n)}+ \frac{\bar S_{B_V} (\bar \cL^{n+1})}{(n+1)(L^n)},\]
    where $(\bar \cV, \bar \cB_\cV; \bar \cL)$ is the compactification of the test configuraiton over $\PP^1$, and \[\bar S_{B_V} = \frac{n(-K_{(V, B_V)}\cdot L^{n-1})}{(L^n)}.\]
\end{thm}
\begin{proof}
    This is obtained by combining the following formulas, each of which is a consequence of equivariant Riemann-Roch or asymptotic Riemann-Roch, see~\cite{BHJ17}:
    \begin{enumerate}
        \item $a_0(\cV_0, \cL_0) = \frac{(L^n)}{n!}, \  a_1(\cV_0, \cL_0) = -\frac{(-K_V\cdot L^{n-1})}{2(n-1)!}$;
        \item $b_0(\cV_0, \cL_0) = \frac{(\bar \cL^{n+1})}{(n+1)!},  \ b_1(\cV_0, \cL_0) = -\frac{1}{2n!} (K_{\bar \cV/\PP^1}\cdot \bar \cL^n)$;
        \item $a_0(\cB_{i,0}, \cL_{i,0}) = \frac{1}{(n-1)!} (B_i\cdot L^{n-1}), \  b_0(\cB_{i,0}, \cL_{i,0})=\frac{1}{n!} (\bar \cB_i\cdot \bar \cL^{n})$.
    \end{enumerate}
\end{proof}

\begin{defn}\label{defn-global-Kstab}
    A log Fano pair $(V, B_V; L = -K_{(V, B_V)})$ is 
    \begin{enumerate}
        \item \emph{K-semistable} if $\fut(\cV, \cB_\cV; \cL)\geq 0$ for all normal ample test configurations;
        \item \emph{K-polystable} if we further have $\fut(\cV, \cB_\cV; \cL)= 0$ only when $(\cV, \cB_\cV; \cL)$ is a product test configuration.
    \end{enumerate}
\end{defn}

\begin{remark}
    The K-stability notion we defined here is in fact what is called $\TT$-equivariant K-stability. We drop the phrase ``$\TT$-equivariant" in this paper for the following two reasons:
    first, we want our terminology to be compatible with the local K-stability story, where only $\TT$-equivariant test configurations are considered, and hence local K-stability is $\TT$-equivariant K-stability already; second,
    by combining the recent works of~\cite{ZhuEquiv, LZFiniteEquiv, LWX, Li19, ZhuangEquiv, XuZhuang20, LXZ}, $\TT$-equivariant (uniform, reduced uniform) K-stablity is equivalent to (uniform, reduced uniform) K-stability for log Fano pairs.
\end{remark}

By using the powerful theory of the Minimal Model Program, to test K-stability for log Fano pairs, we only have to test special test configurations. Here is the precise statement. 
\begin{thm}[\cite{LX14}]
    A log Fano pair $(V, B_V; L= -K_{(V, B_V)})$ is K-polystable if and only if $\fut(\cV, \cB_\cV; \cL)\geq 0$ for all special test configurations, and $\fut(\cV, \cB_\cV; \cL)= 0$ only when $(\cV, \cB_\cV; \cL)$ is a product test configuration.
\end{thm}

\subsection{K-stability of log Fano cone singularities}
In this section, we summarize the parallel K-stability notion for log Fano cone singularities. This was developed in~\cite{CSIrregular}, and we will adapt it to the case of pairs. 

\begin{defn}
	Let $(X, B = \sum_i c_i B_i, \TT; \xi)$ be a log Fano cone singularity. A normal test configuration is the data $(\cX, \cB, \TT, \xi, \eta)$ with a flat projection $\pi: \cX\to \AA^1$ such that
	\begin{enumerate}
		\item $\cX = \Spec \cR$ is a normal affine variety;
		\item $\eta$ is a $\GG_m$ action on $\cX$ lifting the usual action on $\AA^1$;
		\item $\TT$ acts on $\cX$ fiberwise. Further, the action commutes with $\eta$, and coincides with the action on the first factor restricted to $\phi: \cX\times_{\AA^1} (\AA^1\setminus\{0\})\cong X\times (\AA^1\setminus\{0\})$;
        \item If $\cR=\bigoplus_\alpha\cR_\alpha$ is the weight decomposition with respect to $\TT$, then each $\cR_\alpha$ is a flat $\kk[t]$-module;
        \item The Reeb field $\xi$ is chosen fiberwise as given by the same vector in $\TT$, and $\cB = \sum_i c_i \overline{\GG_m\cdot \phi_1^{-1}(B_i)}$, where $\phi_1: \cX_1\cong X$ is the isomorphism over $1\in \AA^1$, and the closure is taken in $\cX$.
	\end{enumerate}
    We will again omit $\TT$ when the torus is clear from context.
\end{defn}

\begin{remark}
    While for most of the NA pluripotential theory, we will only need to work with polarized affine cones, and local K-stability is defined for polarized affine cones in general (see e.g.~\cite{CSIrregular}), we will need the log discrepancy for interesting applications. Thus, we shall restrict ourselves to log Fano cone singularities in this paper. 
\end{remark}

\begin{defn}[compare Definition~\ref{defn-global-tc}]\label{defn-local-tc}
    A test configuration $(\cX, \cB; \xi, \eta)$ is 
    \begin{enumerate}
        \item a \emph{$\QQ$-Gorenstein test configuration} if $K_\cX+\cB$ is $\QQ$-Cartier;
        \item a \emph{weakly special test configuration} if it is $\QQ$-Gorenstein and $(\cX,\cB+\cX_0)$ is lc;
        \item a \emph{special test configuration} if it is $\QQ$-Gorenstein and $(\cX, \cB+\cX_0)$ is plt;
        \item a \emph{product test configuration} if there is a $\TT$-equivariant isomorphism $(\cX, \cB) \cong (X, B)\times \AA^1$, and $\eta$ is of the form $\eta = (\eta',-1)\in N_\RR\oplus \RR$, where $\eta'\in N_\RR$.
    \end{enumerate}
\end{defn}

We need the following extra piece of notation later. Note that the test configuration equips the central fiber with a $\TT\times \GG_m$ action, and $\xi$ is in the Reeb cone $\ft_\RR^{+}$ of $\TT\times \GG_m$. 
Recall that the log discrepancy function on $\ft_\RR^{+}$ is linear, we are going to extend it to a linear function on $N_\RR\oplus \RR$ determined by the same vector in $M_\QQ$ and still denote it by $A$. 

We will now follow~\cite{CSIrregular} to define local K-stability. Let $(X=\Spec R; \xi)$ be a normal affine $\TT$-scheme of dimension $n$ with weight decomposition $R=\bigoplus_\alpha R_\alpha$, and $\xi \in \ft_\RR^+$. Let $\eta\in N_\RR$. We define the index character to be 
\[F(X; \xi, t)\coloneqq \sum_{\alpha} e^{-\la \xi, \alpha\ra t}\dim R_\alpha, \]
and the weight character to be 
\[C_\eta(X; \xi, t)\coloneqq \sum_\alpha e^{-t\la \xi, \alpha\ra}\la \eta, \alpha\ra \dim R_\alpha,\]
for $t\in \CC$ with $\mathrm{Re}(t)>0$.
\begin{thm}[{\cite[Theorem 3, 4]{CSIrregular}}]\label{chap3-thm-index-char}
    The index character and weight character admit meromorphic expansions to a small neighborhood of $0\in \CC$ of the following forms:
    \begin{align*}
        &F(X; \xi, t) = \frac{(n-1)! a_0(X; \xi)}{t^{n}} + \frac{(n-2)! a_1(X; \xi)}{t^{n-1}} +O(t^{-(n-2)}), \\
        &C_\eta(X; \xi, t) = \frac{n! b_0(X; \xi)}{t^{n+1}} + \frac{(n-1)!b_1(X;\xi)}{t^n} + O(t^{-(n-1)}).
    \end{align*}
    Moreover, $a_0(X;\xi), a_1(X; \xi), b_0(X; \xi), b_1(X; \xi)$ depend smoothly on $\xi\in \ft_\RR^+$, and 
    \[b_0(X;\xi) = \frac 1nD_{-\eta} a_0(X;\xi) \coloneqq \frac 1n \dfrac{d}{ds}\Big|_{s=0} a_0(X; \xi-s\eta).\]
\end{thm}
For convenience of comparing the global K-stability definition, we
denote by
\[F_1(X; \xi) = \frac{a_0(X;\xi)b_1(X;\xi)-a_1(X;\xi)b_0(X;\xi)}{a_0(X;\xi)^2}.\]
\begin{defn}[compare Definition~\ref{defn-global-fut}]\label{defn-local-fut}
    Let $(\cX, \cB = \sum_i c_i \cB_i; \xi, \eta)$ be a normal test configuration for $(X, B = \sum_i c_iB_i;\xi)$. Let $(\cX_0, \cB_0 = \sum_i c_i \cB_{i,0}; \xi)$ be the central fiber, and by abuse of notation, still denote by $\xi$ the polarization on $\cB_{i,0}$. The \emph{Futaki invariant} $\fut(\cX, \cB; \xi, \eta)$ of the test configuration is 
    \[\fut(\cX, \cB;\xi, \eta)\coloneqq -2F_1(\cX_0;\xi)+\frac{\sum_ic_i(a_0(\cX_0;\xi)b_0(\cB_{i,0};\xi) - a_0(\cB_{i,0};\xi)b_0(\cX_0;\xi))}{a_0(\cX_0;\xi)^2}.\]
\end{defn}

\begin{remark}
    When $B=0$, this Futaki invariant differs from the one in~\cite{CSIrregular} by a (positive) factor of $\frac{2(n-1)!}{\vol(\xi)}$.
\end{remark}

\begin{defn}[compare Definition~\ref{defn-global-Kstab}]\label{defn-local-Kstab}
    A log Fano cone singularity $(X, B; \xi)$ is 
    \begin{enumerate}
        \item \emph{K-semistable} if $\fut(\cX, \cB; \xi, \eta)\geq 0$ for all normal test configurations;
        \item \emph{K-polystable} if we further have $\fut(\cX, \cB; \xi)= 0$ only when $(\cX, \cB; \xi)$ is a product test configuration.
    \end{enumerate}
\end{defn}
\begin{remark}
    While there is a nice comparison between local and global K-stability for quasi-regular Reeb fields, as we shall describe in Section~\ref{section-local-global-comparison}, there is no intersection-theoretic formula for the local Futaki invariant. This is often the essential difficulty in local K-stability theory. 
\end{remark}

\subsection{Comparison of local and global K-stability}\label{section-local-global-comparison}
Throughout this section, we will assume $(X = \Spec R, B;\xi)$ is a log Fano cone singularity with $\xi = \hat\xi$ rational and primitive. Write $R=\bigoplus_m R_m$ for the weight decomposition induced by $\xi$, and as described in Proposition~\ref{prelim-prop-quotient-log-fano}, let $(V=\Proj \bigoplus_m R_m, B_V; L=-A(\xi)^{-1}(K_V+B_V))$ be the quotient by $\la\xi\ra$.
As we now recall, the global K-stability theory for $(V, B_V; L)$ is the same as the local K-stablity theory for  $(X, B; \xi)$.

\begin{thm}[{\cite[Remark 2.29]{LWX},\cite[Section 2.5]{Li21}}]\label{chap3-thm-tc-comparison}
    There is a one-to-one correspondence between normal (resp. $\QQ$-Gorenstein, weakly special, special) test configurations for $(X, B; \xi)$ and ample normal (resp. $\QQ$-Gorenstein, weakly special, special) test configurations for $(V, B_V; L)$. 
\end{thm}
\begin{proof}
    Given a test configuration $(\cX=\Spec_{\kk[t]} \cR, \cB; \xi, \eta)$, we get a filtration $\cF$ on $R$ via 
    \[\cF^\lambda R_m = \{f\in R_m: t^{-\lambda}\bar f\in \cR_m\},\]
    where $\bar f$ is the lift of $f$ via the $\GG_m$-action. Then the Rees construction, see e.g.\cite[Proposition 2.15]{BHJ17},~\cite[Lemma 2.17]{LWX}, gives an isomorphism 
    \[\cR = \bigoplus_{m\in \NN} \bigoplus_{\lambda\in \ZZ} t^{-\lambda}\cF^\lambda R_m\]
    as finitely generated $\kk[t]$-algebras. Taking $\cV =\Proj_{\kk[t]}\cR$  over $\AA^1$ gives a normal test configuration for $V$, and $\cX = C(\cV, \cL)$ with $\cL = -A(\xi)^{-1}(K_{\cV}+\cB_\cV)$. One can reverse the procedure to get the converse direction. The correspondence between $\QQ$-Gorenstein (resp. weakly special, special) test configurations then follows from similar arguments as in~\cite[Lemma 3.1]{kollarMMP}.
\end{proof}

For a $\QQ$-Gorenstein test configuration $(\cX, \cB; \xi, \eta)$ with $\xi$ rational, one can take $\eta$ to be the canonical lifting of the $\GG_m$ action inducing the corresponding test configuration for the log Fano quotient. Such $\eta$ satisfies $A(\eta)=0$, since by~\cite[Lemma 2.18]{LX18}, $A(\eta) = \frac 1m \frac{\mathscr{L}_\eta s}{s}$ for $s\in |-m(K_\cX+\cB)|$ a $\TT$-equivariant nowhere vanishing global section, and $m$ divisible enough. More generally, we always have the following. 

\begin{thm}\label{thm-local-global-fut}
    Let $(\cX=\Spec_{\kk[t]} \cR, \cB; \xi, \eta)$ be a test configuration for $(X, B;\xi)$, and $(\cV, \cB_\cV; \cL)$ be the corresponding test configuration for $(V, B; L)$. Then 
    \[\fut(\cX, \cB; \xi, \eta) = \fut(\cV, \cB_\cV; \cL).\]
\end{thm}
\begin{proof}
    This is a consequence of the orbifold Riemann-Roch, as in~\cite[Section 2.8, 2.9]{RossThomas}, together with~\cite[Proposition 4.1, 4.4]{CSIrregular}. All the defining coefficients $a_i, b_i$ for $i=0,1$, in this case have the same formulas as in the proof of Theorem~\ref{thm-global-intersection}, except that $a_1(\cV_0, \cL_0) = -\frac{(-K_V^{\mathrm{orb}}\cdot L^{n-1})}{2(n-1)!}$, and $b_1(\cV_0, \cL_0) = -\frac{1}{2n!}(K_{\bar\cV/\PP^1}^{\mathrm{orb}}\cdot \bar\cL^n)$,
    where $K_V^{\mathrm{orb}}$ and $K_{\bar\cV}^{\mathrm{orb}} $ denote the orbifold canonical divisors. 
\end{proof}

\begin{remark}\label{rmk-local-global-fut}
    In view of the previous theorem, if one thinks of $(\cV, \cB_\cV; A_{(X,B)}(\xi)\cL)$ as a test configuration for $(V, B_V; -K_{(V,B_V)})$, then we have $\fut(\cX, \cB; \xi, \eta) = \fut(\cV, \cB_\cV; A_{(X, B)}(\xi)\cL)$. Furthermore, it is not hard to see from Theorem~\ref{chap3-thm-index-char} and the definition that the Futaki invariant is homogeneous in $\xi$ of degree $0$, and translation invariant, i.e., $\fut(\cX,\cB;\xi, \eta+c\xi) = \fut(\cX,\cB;\xi, \eta)$ for $c>0$.
\end{remark}

An important observation in~\cite{BHJ17} about test configurations for log Fano pairs is that each test configuration corresponds to a finite collection of divisorial valuations on the log Fano pair, and in particular, a weakly special test configuration corresponds to a finitely generated divisorial valuation (in fact, they are lc places of a complement, see e.g.~\cite{BLX}). There are analogous results regarding test configurations for log Fano cone singularities, which we shall summarize as the following theorem.

\begin{thm}[{\cite[Lemma 2.21]{LWX}}]\label{chap3-thm-local-tc-char}
    A test configuration $(\cX, \cB;\xi, \eta)$ for $(X, B;\xi)$, with $\eta$ in the Reeb cone of $\TT\times \GG_m$, corresponds to a $\TT$-invariant Weil divisor $E$ with integer coefficients over $o\in X$ with $-E$ ample. In particular, a special test configuration yields a Koll\'ar component.
\end{thm}
\begin{remark}
    The condition that $\eta$ is in the Reeb cone can be achieved by twisting $\eta$ by some rational vector in the Reeb cone. 
\end{remark}

\subsection{Torus equivariant complements}
It is by now well-known that (weakly) special test configurations of log Fano varieties correspond to lc places of complements, and we now derive a local version for log Fano cone singularities. 
\begin{defn}[global complements]
    Let $(V, B_V)$ be a log Fano pair. A \emph{$\QQ$-complement (resp. $\RR$-complement)} of $(V, B_V)$ is an effective $\QQ$-divisor (resp. $\RR$-divisor) $D$ on $V$ such that $D\geq B_V$, $(V, D)$ is lc and $K_V+D\sim_\QQ 0$ (resp. $K_V+D\sim_\RR 0$). An \emph{$N$-complement} of $(V, B_V)$ is a $\QQ$-complement $D$ with $N(K_V+D)\sim 0$.
\end{defn}
\begin{defn}[local complements]
    Let $(X, B; o)$ be a klt singularity. A \emph{(local) $\QQ$-complement} of $(X, B; o)$ is an effective $\QQ$-divisor $\Gamma$ on $X$ such that $\Gamma \geq B$, $(X, \Gamma)$ is lc and $o$ is an lc center. A \emph{(local) $N$-complement} of $(X, B; o)$ is a $\bQ$-complement $\Gamma$ with $N(K_X+\Gamma)\sim 0$. 
\end{defn}

We note that complements in this paper are always assumed to be monotonic, which differ slightly from the original definition (see e.g. \cite[Definition 3.17]{HLS19} for the terminology).

The following result generalizes \cite[Theorem A.2]{BLX} to the torus equivariant setting. 

\begin{thm}\label{thm:torus-comp}
Let  $I\subset[0,1]\cap\bQ$ be a finite set. 
Let $(V,B_V)$ be a log Fano pair with an algebraic torus $\bT$-action and $\Coeff(B_V)\subset I$. Let $E$ be a $\bT$-equivariant prime divisor over $V$.
Then there exists a positive integer $N$ depending only on $\dim(V)$ and $I$ such that the following are equivalent.

\begin{enumerate}
    \item $E$ induces a $\bT$-equivariant weakly special test configuration of $(V,B_V)$ with integral central fiber;
    \item $E$ is an lc place of a $\bQ$-complement of $(V,B_V)$;
    \item $E$ is an lc place of an $
    \bR$-complement of $(V,B_V)$;
    \item $E$ is an lc place of an $N$-complement of $(V,B_V)$;
    \item $E$ is an lc place of a $\bT$-equivariant $N$-complement of $(V,B_V)$.
\end{enumerate}
\end{thm}

\begin{proof}
The equivalence between~(1),~(2) and~(4) was proved in~\cite[Theorem A.2]{BLX} based on boundedness of complements from~\cite{Birkar}. Clearly~(2) implies~(3), and~(5) implies~(4). Thus we shall focus on showing~(3) $\Rightarrow$ ~(2) and~(4) $\Rightarrow$~(5). 
By multiplying the denominators of elements in $I$, we may assume that $I \subset \frac{1}{N}\bZ$.

We first show~(3) $\Rightarrow$~(2), which should be well-known. Let $D$ be an $\bR$-complement of $(V,B_V)$ such that $A_{V,D}(E) = 0$. By~\cite[Theorem 5.6]{HLS19}, by perturbing coefficients of $D$ into rational numbers, there exist $\bQ$-divisors $D_1, \cdots, D_l$ and positive real numbers $c_1,\cdots, c_l$ such that for each $1\leq i\leq l$, $\supp(D_i) = \supp(D)$, $(X, D_i)$ is lc, and $D=\sum_{i=1}^l c_i D_i$ with $\sum_{i=1}^l c_i = 1$. Moreover, we may assume that $D_i \geq B_V$ for each $i$. Thus 
\[
0 = A_{V,D}(E) = \sum_{i=1}^l c_i A_{V, D_i}(E) \geq 0
\]
implies that $E$ is an lc place of $(V,D_i)$ for each $1\leq i\leq l$. Since $K_V + D = K_V + \sum_{i=1}^l c_i D_i\equiv 0$, we can perturb coefficients $c_i$ to $c_i'\in \bQ_{>0}$ such that $\sum_{i=1}^l c_i' =1$ and $K_V + \sum_{i=1}^l c_i' D_i\equiv 0$. Since $V$ is rationally connected, we know that $D':=\sum_{i=1}^l c_i' D_i$ is a $\bQ$-complement of $(V,B_V)$ such that $A_{V,D'}(E) = 0$. This shows~(3) $\Rightarrow$~(2).

For the remaining part, we show~(4) $\Rightarrow$~(5).
Since~(4) implies~(1), we have that $E$ induces a weakly special test configuration $(\cV, \cB_{\cV})$ of $(V,B_V)$ with integral central fiber. Moreover, we know that $(\cV, \cB_{\cV})$  is $\bT$-equivariant as $E$ is $\bT$-equivariant.
Let $(V_0, B_0)$ be the central fiber of $(\cV, \cB_{\cV})$  equipped with an action of $\bT\times \bG_m$. Denote by $\sigma$ the $1$-PS of $\bT\times \bG_m$ coming from the second component, i.e. $\sigma(t) = (1, t)$. 
Let $\lambda: \bG_m \to \bT\times \bG_m$ be an arbitrary non-trivial $1$-PS. Denote by $\lambda_k:=k\sigma + \lambda$. We first show that there exists a $\lambda_k$-equivariant $N$-complement of $(V_0, B_0)$ for $k\gg 1$. Our argument is similar to~\cite[Remark 4.11]{Xusurvey}.
Since $\lambda$ induces a product special test configuration $(\cV_\lambda, \cB_{\cV, \lambda})$ of $(V_0, B_0)$, following~\cite[Proof of Lemma 3.1]{LWX} we may combine the two test configurations $\cV$ and $\cV_{\lambda}$  to obtain a weakly special test configuration $(\cV_{\lambda_k}, \cB_{\cV,\lambda_k})$ of $(V,B_V)$ whose central fiber is isomorphic to $(V_0, B_0)$ with the $\bG_m$-action $\lambda_k$ as long as $k\gg 1$. Applying (1) $\Rightarrow$ (4) to $(\cV_{\lambda_k}, \cB_{\cV,\lambda_k})$ yields an $N$-complement $\Delta_{\lambda_k}$ of $(V,B_V)$ and an lc place $E_{\lambda_k}$ of $(V, \Delta_{\lambda_k})$ that induces $(\cV_{\lambda_k}, \cB_{\cV,\lambda_k})$.
Let $\Delta_{\tc, \lambda_k}$ be the Zariski closure of $\Delta_{\lambda_k}\times (\bA^1\setminus\{0\})$ under the identification $\cV_{\lambda_k} \times_{\bA^1} (\bA^1\setminus \{0\})\cong V\times (\bA^1\setminus\{0\})$. Since $N(K_{V}+\Delta_{\lambda_k})\sim 0$, we know that $N(K_{\cV_{\lambda_k}}+\Delta_{\tc,\lambda_k})\sim 0$ as the central fiber $\cV_{\lambda_k, 0}\cong V_{0}$ is integral. Let $\Delta_{\lambda_k,0}$  be the restriction of $\Delta_{\tc,\lambda_k}$ to the central fiber $\cV_{\lambda_k, 0}$. Thus we have $N(K_{V_0}+\Delta_{\lambda_k,0})\sim 0$. By~\cite[Theorem 4.8]{ABB+} (cf.~\cite{CZ21}), we have that $(V_0, \Delta_{\lambda_k,0})$ is slc which implies that $\Delta_{\lambda_k,0}$ is a $\lambda_k$-invariant $N$-complement of $(V_0, B_0)$. This works as long as $k\geq k_{\lambda}$ for some $k_{\lambda}\in \bN$  depending only on $\lambda$.

Next, we show that there exists a $\bT\times \bG_m$-invariant $N$-complement $\Delta_0$ of $(V_0, B_0)$.  
We know that the collection of $N$-complements form an open subset of the projective space $\bP(H^0(V, \cO_V(-N(K_V+B_V))))$ as $NB_V$ is a $\bZ$-divisor. Thus Lemma~\ref{lem:torus-stabilizer} below implies that the set of identity component of  stabilizers 
\[
\{\mathrm{Stab}^0(\Delta_{\lambda_k,0})\mid \lambda\in \mathrm{Hom}(\bG_m, \bT\times\bG_m)\setminus\{1\}, ~ k \geq k_\lambda\}
\]
is a finite set of subtori $\{\bT_i \leq \bT\}_{i=1}^l$ of $\bT\times\bG_m$. Since $\Delta_{\lambda_k,0}$ is $\lambda_k$-equivariant, we know that $\cup_{i=1}^l \bT_i$ contains the image of $\lambda_k=k\sigma+\lambda$ for every non-trivial $1$-PS $\lambda$ of $\bT\times\bG_m$ and every $k\geq k_{\lambda}$. Since there are infinitely many $k$ but only finitely many subtori $\bT_i$, we know that there exists $i$ such that the image of $\lambda$ is contained in $\bT_i$. Since $\lambda$ is arbitrary, this implies that there exists $\lambda$, $k$ and $i$ such that $\bT_i = \mathrm{Stab}^0(\Delta_{\lambda_k,0}) = \bT\times\bG_m$. Hence $\Delta_0:=\Delta_{\lambda_k,0}$ is a $\bT\times\bG_m$-invariant $N$-complement of $(V_0, B_0)$.

Finally, we show that one can deform $\Delta_0$ to a $\bT$-invariant complement $\Delta$ of $(V, B_V)$ such that $E$ is an lc place. Let $\Gamma$ be an $N$-complement of $(V,B_V)$ such that $(\cV, \cB_{\cV})$ is induced by an lc place $E$ of $(X,\Gamma)$. Let $\Gamma_{\tc}$ be the Zariski closure of $\Gamma\times (\bA^1 \setminus\{0\})$ under the isomorphism $\cV\times_{\bA^1} (\bA^1\setminus\{0\}) \cong V\times (\bA^1 \setminus\{0\})$. By~\cite[Theorem 4.8]{ABB+} we know that $(\cV, \cB_{\cV} + (\Gamma_{\tc}-\cB_{\cV}))\to \bA^1$ is a family of boundary polarized CY pairs in the sense of~\cite[Definition 2.9]{ABB+}.
 Since $(\cV, \cB_{\cV})\to \bA^1$ is $\bT\times\bG_m$-equivariant, applying~\cite[Lemma 12.2]{ABB+} to the family $(\cV,\Gamma_{\tc})\to \bA^1$ implies that there exists a $\bT\times\bG_m$-equivariant divisor $\Delta_{\tc}$  on $\cV$ such that $(\cV, \cB_{\cV}^+)\to \bA^1$ is a $\bT\times \bG_m$-equivariant family of boundary polarized CY pairs with $\Delta_{\tc}|_{V_0}=\Delta_0$ and $N(K_{\cV}+\Delta_{\tc})\sim 0$. Thus by~\cite[Theorem 4.8]{ABB+} we know that $\Delta:=\Delta_{\tc}|_{X\times \{1\}}$ is a $\bT$-equivariant $N$-complement of $(V,B_V)$ such that $E$ is an lc place. The proof is finished.
\end{proof}

\begin{lem}\label{lem:torus-stabilizer}
Let $W$ be a $\bC$-vector space admitting a linear $\bT$-action which induces an action on the projective space $\bP(W)$. Then there exists a finite set of subtori $\{\bT_i \leq \bT\}_{i=1}^l$ such that for every $[w]\in \bP(W)$, there exists $1\leq i\leq l$ satisfying the identity component of the stabilizer $\mathrm{Stab}^0([w]) = \bT_i$. 
\end{lem}

\begin{proof}
Let $W= \bigoplus_{\chi \in M}W^\chi$ be the weight decomposition where $M= \mathrm{Hom}(\bT, \bG_m)$ is the character space. Denote by $\{\chi_1, \cdots, \chi_m\}= \{\chi\in M \mid W^\chi \neq 0\}$. Then for every $w\in W$ we can write $w= \sum_{j=1}^m c_j w_j$ where $c_j\in \bC$ and $w_j \in W^{\chi_j}$. Since every connected algebraic subgroup of the torus is a subtorus, we know that the identity component of the stabilizer subgroup $\mathrm{Stab}^0([w])$ is a subtorus of $\bT$. Hence we know that 
\[
\mathrm{Stab}^0([w]) = \{t^\alpha \mid \alpha \cdot \chi_j = \alpha \cdot \chi_k \textrm{ if } c_j  \neq 0 \textrm{ and } c_k \neq 0\}.
\]
Since $\chi_j$ has finitely many choices, we know that
 $\mathrm{Stab}^0([w])$ has only finitely many choices. 
\end{proof}

As a consequence, the following result generalizes the existence of bounded complements for log Fano pairs from \cite{Birkar} to the $\bT$-equivariant setting.

\begin{cor}
Let $n$ be a positive integer and $I\subset[0,1]\cap\bQ$ be a finite set. 
 Then there exists a positive integer $N$ depending only on $n$ and $I$ such that for any $(n-1)$-dimensional klt log Fano pair $(V,B_V)$  with an algebraic torus $\bT$-action and  $\Coeff(B_V)\subset I$, there exists a $\bT$-equivariant $N$-complement of $(V,B_V)$.
\end{cor}

\begin{proof}
We may assume that the $\bT$-action is non-trivial as otherwise it follows from~\cite{Birkar}. Thus there exists a non-trivial product test configuration $(\cV, \cB_{\cV})$ of $(V,B_V)$ induced by some $1$-PS in $\bT$. Then by Theorem~\ref{thm:torus-comp} there exists a $\bT$-equivariant $N$-complement $\Delta$ of $(V,B_V)$ and an lc place $E$ of $(X,\Delta)$ that induces $(\cV, \cB_{\cV})$. The proof is finished.
\end{proof}

Theorem~\ref{main-thm-complements} is now a consequence of Theorem~\ref{thm:torus-comp}.

\begin{cor}[= Theorem \ref{main-thm-complements}]
Let $n$ be a positive integer and $I\subset [0,1]\cap \QQ$ be a finite set. Then there exists a positive integer $N$ depending only on $n$ and $I$ such that
for any $n$-dimensional log Fano cone singularity $(X, B, \TT=\GG_m^r; \xi)$ with $\Coeff(B)\subset I$ and $E$ a $\bT$-equivariant Koll\'ar component giving rise to a nontrivial special test configuration of $(X, B; \xi)$, there exists a $\TT$-equivariant $N$-complement $\Gamma$ of $(X, B, o)$ such that $E$ is an lc place of $(X, \Gamma)$.
\end{cor}
\begin{proof}
    We first show that there exists a $\bT$-equivariant $\bQ$-complement $\Gamma_1$ of $(X, B, o)$ such that $E$ is an lc place of $(X,\Gamma_1)$. Since special test configurations are preserved after changing the Reeb vector field, we may replace $\xi$ by a primitive integer vector in the Reeb cone. Denote by $(V(\xi), B_V(\xi))$ the resulting log Fano pair via taking the quotient of $(X,B)$ by $\langle\xi\rangle$ by Proposition \ref{prelim-prop-quotient-log-fano}. Then $(V(\xi), B_V(\xi))$ admits a $\TT(\xi)\cong \GG_m^{r-1}$ action, and $E$ induces a special test configuration of $(V(\xi), B_V(\xi))$. Moreover, $\ord_E = (t\xi)*v_F$ for some $t>0$ and $v_F\in V(\xi)^{\an}$ a special divisorial valuation in the sense of~\cite{BLX}.
    By Theorem~\ref{thm:torus-comp}, there is a $\TT(\xi)$-invariant $\bQ$-complement $D_1$ of $(V(\xi), B_V(\xi))$ such that $v_F$ is an lc place of $(V(\xi), D_1)$. 
    Take $\Gamma_1:= B + \overline{\pi^* (D_1-B_V(\xi))}$ where $\pi: X\setminus\{o\}\to V(\xi)$ is the Seifert $\bG_m$-bundle.  Then we know that the quotient of $(X, \Gamma_1)$ by $\langle \xi\rangle$ is precisely $(V(\xi), D_1)$ which is an lc log Calabi-Yau pair. Thus $\Gamma_1$ is a $\TT$-invariant $\QQ$-complement of $(X,B; o)$, and $E$ is an lc place. 

    Next, we take the projective compactification $\oX$ of $X$ as a projective orbifold bundle over $V(\xi)$ (see e.g. \cite[Section 3.1]{ZhuangboundednessII}). Let $\oB$ and $\oGamma_1$ be the closure of $B$ and $\Gamma_1$ in $\oX$, respectively. Denote by $V_\infty$ the section at infinity of $\oX$. Thus by \cite[Lemma 3.3]{ZhuangboundednessII} we know that $(\oX, \oB)$ is a klt log Fano pair. Moreover, it has a $\bQ$-complement $\oGamma_1+V_\infty$ and an lc place $E$. Applying Theorem~\ref{thm:torus-comp} to $E$ over $(\oX, \oB)$, we know that there exists $N$ depending only on $n$ and $I$ and a $\bT$-equivariant $N$-complement $\oGamma$ of $(\oX, \oB)$ such that $E$ is an lc place of $(\oX, \oGamma)$. Thus taking $\Gamma:=\oGamma|_{X}$ finishes the proof.
\end{proof}

\section{A valuative criterion for K-semistability}\label{section: localval}
In this section, we briefly recall the valuative criterion studied in~\cite{HuangThesis} in our notation, and deduce a combinatorial criterion for checking K-semistability of toric varieties.

For $(X=\Spec R, B; \xi)$ a log Fano cone singularity, we define the \emph{volume} of a $\TT$-invariant valuation $v$ on $X$ to be the volume of the associated filtration $\cF_v$. More precisely, let $R_m\coloneqq \bigoplus_{\la\xi, \alpha\ra\leq m} R_\alpha$, and let $(a_{m, j}, 1\leq j\leq \dim R_m)$ be the jumping numbers of $\cF_v$ in $R_m$. Then 
\[
S(v;\xi)\coloneqq \lim_{m\to \infty} \frac{1}{m\dim R_m}\sum_j a_{m, j}.
\]
We refer the readers to~\cite{XuZhuang20, wu} for more details on this invariant.
\begin{defn}
    We set 
    \[\delta(X, B; \xi) =\frac{n}{(n+1)A(\xi)}\inf_v \frac{A_{(X, B)}(v)}{S(v;\xi)},\]
    where $v$ runs through all nontrivial $\TT$-invariant valuations in $X_0$ with finite log discrepancy.
\end{defn}

\begin{remark}
    We note that a similar volume invariant, which in this section we will denote by $S'(v;\xi)$, is defined in~\cite{HuangThesis} as the average of jumping numbers in a slightly different setting:
    \[S'(v;\xi)\coloneqq \lim_{m\to \infty} \frac{A(\xi)}{m\dim R_m'}\sum_j a_{m,j},\] where 
    $R_m'\coloneqq \bigoplus_{\la \xi, \alpha\ra\in (m-1, m]} R_\alpha$, and $a_{m, j}$'s are jumping numbers of $\cF_v$ in $R_m'$. 
    In \emph{loc.cit.} the corresponding $\delta$-invariant, which we will denote by $\delta'$, is defined by
    \[\delta'(X, B; \xi)\coloneqq \inf_{v} \frac{A_{(X, B)}(v)}{S'(v;\xi)},\]
    where $v$ runs through all $\TT$-invariant valuations centered at $o$ with finite log discrepancy.
\end{remark}
We first relate the two $\delta$-invariants.
\begin{prop}
    We have the following equality
    \[\delta'(X, B; \xi) = \min\{1, \delta(X, B; \xi)\}.\]
    Moreover, when $\delta(X, B; \xi)\leq 1$, there is a valuation in $X_0$ computing $\delta(X, B; \xi)$.
\end{prop}
\begin{proof}
    For simplicity of notation, we will write $S(v)$ and $S'(v)$ for $S(v;\xi)$ and $S'(v;\xi)$. We first claim that 
    \[S'(v) = \frac{A(\xi)(n+1)}{n}S(v),\]
    for all $v\in X_0^{\TT}$. We remark that this in particular shows that $\delta(X, B; \xi)$ is the same as the one defined in~\cite{XuZhuang20}.
    Indeed, by~\cite[Lemma 3.11]{wu}, we have that when $\xi$ is rational and primitive, $S'(v) = A(\xi)\tilde{S}(\cF_v; L) = A(\xi)\frac{n+1}{n}S(v)$, where $A(\xi)L\sim_\QQ -K_{(V, B_V)}$, for $(V, B_V)$ the quotient log Fano pair. For $\xi$ general, the equality follows from continuity and homogeneity of both sides in $\xi$.
    The second observation is a translation property. Let $w=(t\xi)*v$. Then we have 
    \[S'(w) = S'(v)+tA(\xi).\] 
    This follows by noting $\cF_w^\lambda R_\alpha = \cF_v^{\lambda- t\la\xi, \alpha\ra} R_\alpha$ for each $\alpha\in \Lambda$, and so the jumping numbers shift by exactly $t\la\xi, \alpha\ra$.

    Now for any non-trivial $\TT$-invariant valuation $v$ in $X_0$, consider the function $f:[0,\infty)\to \RR$ defined by 
    \[f(t)\coloneqq \frac{A_{(X, B)}(v)+tA(\xi)}{S(v)+tA(\xi)} = \frac{A_{(X, B)}(w_t)}{S'(w_t)},\]
    where $w_t = (t\xi)*v$, and the equality follows from Lemma~\ref{logdiscreptranslation}. Note
    \[f'(t) = \frac{A(\xi)(S(v)-A_{(X, B)}(v))}{(S(v)+tA(\xi))^2}.\]
    Hence $f$ is decreasing if $S(v)<A_{(X, B)}(v)$; and increasing if $S(v)>A_{(X, B)}(v)$. 

    Thus if $\delta(X, B;\xi)<1$, then $f$ is increasing for any $v$ with $S(v)>A_{(X, B)}(v)$, and so $\delta'(X, B; \xi) = \delta(X, B; \xi)$; if $\delta(X, B;\xi)>1$, then $f$ in decreasing for any $v$, and in this case $\delta'(X, B; \xi)=1$ by taking $t\to \infty$. This proves the equality on $\delta$-invariants. Finally, it was shown in~\cite[Theorem 5.0.4]{HuangThesis} that if $\delta'(X, B; \xi)\le 1$, then there is a $\TT$-invariant valuation computing $\delta'(X, B; \xi)$, and the second assertion now directly follows from the equality.
\end{proof}

Thanks to the characterization of K-semistability through the minimizer of the normazlied volume function, it is shown in~\cite{HuangThesis} (see also~\cite{XuZhuang20}) that the $\delta$-invariant introduced above is a valuative criterion for K-semistability.
\begin{thm}[{\cite[Theorem 1.0.3]{HuangThesis}}]
    A log Fano cone singularity $(X, B; \xi)$ is K-semistable with respect to special test configurations if and only if $\delta(X, B;\xi)\geq 1$.
\end{thm}



We now use this to give a criterion for K-semistability of toric varieties. This will recover a result in~\cite{BermanToric} via a different approach. We first set up some notation. Let $X=X_\sigma$ for some convex cone $\sigma\in \RR^n$ generated by primitive vectors $v_1, \cdots, v_d$. Then $X$ is $\QQ$-Gorenstein if and only if there is some $l\in \sigma^\vee\cap M_\QQ$ such that $\la v_i, l\ra=1$, see e.g.~\cite[Proposition 3.2]{BermanToric} for a proof. Since $\delta$ does not depend on the scaling of $\xi$, we will normalize $\xi$ by $\la \xi, l\ra=1$. Let $Q_\xi=\{u\in \sigma^\vee: \la\xi, u\ra\leq 1\}, P_\xi=\{u\in \sigma^\vee: \la\xi, u\ra= 1\}$.

\begin{thm}
    Let $(X; \xi)$ be an affine toric variety described above. Then 
    \[\delta(X; \xi) = \min_{i=1,\cdots, d} \frac{1}{\la v_i, \bar u^P\ra} = \frac{n}{n+1}\min_{i=1,\cdots, d}\frac{1}{\la v_i, \bar u^Q\ra},\]
    where $\bar u^P, \bar u^Q$ are barycenters of $P_\xi, Q_\xi$ respectively.
\end{thm}
\begin{proof}
    First note that $\TT$-invariant valuations in $X_0$ correspond to vectors on the boundary of $\sigma$. Let $v$ be a toric valuation centered on $X$, that is $v\in \sigma$. Then
    \[S_m(v) = \frac{1}{m\#(mQ_\xi\cap M)}\sum_{u\in mQ_\xi\cap M} \la v, u\ra = \la v, \bar u_m^Q\ra,\]
    where $\bar u_m^Q$ denotes the barycenter of $Q_\xi\cap m^{-1}M$.
    Thus \[S(v) = \la v, \bar u^Q\ra.\]
    A similar argument shows that
    \[A(\xi)^{-1}S'(v) = \la v, \bar u^P\ra.\] 
    One immediately sees that the factor $\frac{n+1}{n}$ relating $S$ and $S'$ comes from the fact that $\bar u^P = \frac{n+1}{n}\bar u^Q$. It thus suffices to prove the first equality.
    
    Note also that $A(\xi) = \la l, \xi\ra=1$. We have 
    \[\delta(X;\xi) = \inf_{v\in \partial \sigma}\frac{\la v, l\ra}{\la \xi, l\ra\la v, \bar u^P\ra} = \min_{i=1, \cdots, d} \frac{\la v_i, l\ra}{\la v_i, \bar u^P\ra} =\min_{i=1, \cdots, d} \frac{1}{\la v_i, \bar u^P\ra}. \]
\end{proof}
\begin{cor}[compare {\cite[Proposition 2.4]{BermanToric},~\cite[Corollary 7.17]{BlumJonsson}}]
    A toric log Fano cone singularity $(X; \xi)$ with $\la \xi, l\ra=1$ is K-semistable if and only if $\bar u^P=l$.
\end{cor}
\begin{proof}
    Since $\delta(X; \xi) = \min_{i=1, \cdots, d} \frac{1}{\la v_i, \bar u^P\ra}$, we have $\delta(X; \xi)\leq 1$, with equality achieved only when $\bar u^P=l$.
\end{proof}

\section{Non-Archimedean functionals}\label{section: NAfunctional}
We now introduce the non-Archimedean functionals needed for characterizing local K-stability. Most of the functionals are defined using mixed Monge--Amp\`ere measures developed in~\cite{wu}. Throughout, for a fixed test configuration $(\cX, \xi, \eta)$, we will write $\chi_\eta$ for the NA norm corresponding to the filtration defined by the test configuration, and the associated FS function is $\varphi=\fs(\xi, \chi_\eta)$.

\subsection{The non-Archimedean Mabuchi functional}
We first recall the following definition of the Monge--Amp\`ere energy for a FS function $\varphi$, which agrees with the one in~\cite{wu} up to a multiple of $A(\xi)$.
\begin{defn}
The non-Archimedean Monge--Amp\`ere energy $E^{\na}$ is defined by
\[E^{\na}(\varphi) = \frac{A(\xi)}{n}\sum_{i=0}^{n-1}\int_{X^{\an}\setminus\{0\}} \varphi \ma(\varphi^{[i]}, \varphi_{\xi}^{[n-1-i]})\]
where
\[\ma(\varphi^{[i]}, \varphi_{\xi}^{[n-1-i]}):=\frac{1}{\vol(\xi)}  (d'd''\varphi)^j\wedge (d'd''\varphi_\xi)^{n-j-1}\wedge d'd''\varphi_\xi^+,\]
and we will write
$\ma(\varphi):= \ma(\varphi^{[n-1]})$ for the unmixed Monge--Amp\`ere measure.
\end{defn}

The following explains the choice of this normalization in the context of log Fano cone singularities. 
\begin{lem}\label{normalization}
Suppose $\varphi$ comes from a test configuration $(\mathcal{X}, \mathcal{B}; \xi, \eta)$ and $\xi$ is quasi-regular. Let $(\mathcal{V}, \mathcal{B}_\cV; \cL)$ be the corresponding test configuration for the quotient log Fano pair $(V, B_V) = ((X, B)\setminus\{o\})/\la \xi\ra$ inducing a non-Archimedean metric $\phi$ on $-(K_V+B)$. Then
\[E^{\na}(\varphi) = E^{\na}(\phi).\]
\end{lem}
\begin{proof}
Let $\tilde{\phi}=A(\xi)^{-1}\phi$ be the metric on $L\sim_\QQ -A(\xi)^{-1}(K_V+B_V)$ induced by the fiberwise quotient of $(\cX, \cB;\xi, \eta)$.
When $\xi$ is the fiber coordinate of $L$, i.e., when $\xi$ is primitive, we have 
\[ E^{\na}(\varphi) = A(\xi)E^{\na}(\tilde{\phi}) = E^{\na}(A(\xi)\tilde{\phi)} = E^{\na}(\phi),\] 
where the fist equality follows from~\cite[Corollary 5.14]{wu}.
Note that the left hand side is independent of the scaling of $\xi$, we can conclude.
\end{proof}

In a similar manner, we define the other functionals using the same normalization.
\begin{defn}
The non-Archimedean functionals $I^{\na}, J^{\na}$ for $\varphi\in \cH(\xi)$ are defined by 
\[I^{\na}(\varphi)\coloneqq A(\xi)\int_{X^{\an}} \varphi \mathrm{MA}(\varphi_\xi)
    - A(\xi)\int_{X^{\an}} \varphi \mathrm{MA}(\varphi),\]
and
\[J^{\na}(\varphi)\coloneqq A(\xi)\int_{X^{\an}} \varphi \mathrm{MA}(\varphi_\xi) - E^{\na}(\varphi).\]
\end{defn}

\begin{defn}
Let $\varphi\in\cH(\xi)$. The \emph{non-Archimedean entropy functional} is 
\[H^{\na}(\varphi):= \int_{X^{\an}} A(v) \mathrm{MA}(\varphi), \]
and we define the \emph{non-Archimedean Mabuchi energy} by
\[M^{\na}(\varphi):= H^{\na}(\varphi)-(I^{\na}-J^{\na})(\varphi).\]
\end{defn}

\begin{prop}\label{chap5-prop-local-global-mabuchi}
If $\varphi$ corresponds to a FS metric $\phi$ on $-(K_V+B_V)$ for a log Fano variety $(V, B_V)$, whose cone is $(X, B; \xi)$, then we have
\[I^{\na}(\varphi) = I^{\na}(\phi), \
  J^{\na}(\varphi) = J^{\na}(\phi), \
  H^{\na}(\varphi) = H^{\na}(\phi),\]
and
\[M^{\na}(\varphi) = M^{\na}(\phi)\]
\end{prop}
\begin{proof}
This is a consequence of~\cite[Proposition 5.13]{wu}, and follows from a similar argument as in Lemma~\ref{normalization}. Indeed, when $\xi$ is primitive, we have
\[F^{\na}(\varphi) = F^{\na}(A(\xi)\varphi)= F^{\na}(\phi)\]
for $F\in \{I, J, H, M\}$.
The identities then follow by noting the the left hand side is homoegenous in $\xi$ of degree $0$.
\end{proof}
Part of Theorem~\ref{thm-main-cts} is the following.
\begin{thm}\label{chap5-thm-mabuchi-cts}
    Fix a norm $\chi$ of finite type. For any $F\in \{E, I, J, H, M\}$, the function
    \[\xi\mapsto F^{\na}(\fs(\xi, \chi))\]
    is continuous in the Reeb cone.
\end{thm}
\begin{proof}
    For $F\in \{E, I, J\}$, this is a consequence of the continuity of mixed Monge--Amp\`ere measures as shown in Corollary~\ref{ctsMAsupp}. 

    For the entropy functional, this is a consequence of Lemma~\ref{logdiscreptranslation}, Lemma~\ref{MeasureTranslation} and Lemma~\ref{Aconvergence}. Indeed, combining these lemmas and writing in the notation thereof, we have
    \begin{align*}
    H^{\na}(\fs(\xi_i, \chi))+tA(\xi_i) &= \int_{X^{\an}} A((t\xi_i)*v)\mu_0^i = \int_{X^{\an}} A(v)\mu_t^i \\
    &\to \int_{X^{\an}} A(v)\mu_t = H^{\na}(\fs(\xi, \chi))+tA(\xi)
    \end{align*}
    for any $t>0$ and $\xi_i\to \xi$.
    Finally the continuity of the Mabuchi functionals follows by combining the continuity properties of $H, I, J$.
\end{proof}

\subsection{The non-Archimedean Ding functional}
\begin{defn}
Fix $\varphi\in\cH(\xi)$. We define the non-Archimedean functional 
\[L^{\na}(\varphi) = \inf_v \{A_{(X,D)}(v) + A(\xi)\varphi(v)\},\]
where $v$ runs over all $\mathbb{T}$-invariant quasi-monomial valuations in $X_0$.
The \emph{non-Archimedean Ding functional} is defined by
\[D^{\na}(\varphi) := L^{\na}(\varphi) - E^{\na}(\varphi).\]
\end{defn}

\begin{remark}
We remark that this definition of $L^{\na}$ stays the same if we let $v$ run through all $\mathbb{T}$-invariant quasi-monomial valuations centered on $X$. Indeed, any such $w$ can be written as $w= \lambda \xi *v$ for some $\lambda \geq 0, v\in X_0^\TT.$ Then by Lemma~\ref{logdiscreptranslation}, we have that $A(w) = \lambda A(\xi)+A(v)$, and also note that $A(\xi)\varphi(w) = A(\xi)\varphi(v)-\lambda A(\xi)$. Thus we only have to take the inf over $X_0.$
\end{remark}
\begin{prop}\label{chap5-prop-local-global-ding}
Assume $\varphi$ corresponds to a FS metric $\phi$ on a log Fano pair $(V, B_V)$, whose cone is $(X, B; \xi)$. Then we have 
\[L^{\na}(\varphi) = L^{\na}(\phi).\]
\end{prop}
\begin{proof}
The proof is again similar to the proof of Lemma~\ref{normalization}. 
 When $\xi$ is primitive, we have that $\varphi = \phi'$ for some NA metric $\phi'$ on $L\sim_\QQ -A(\xi)^{-1}(K_V+B_V)$. Thus in this case $A(\xi)\varphi = A(\xi)\phi' = \phi$ is a NA metric on $-(K_V+B)$. In general, note that $A(\xi)\varphi$ is homogeneous in $\xi$ of degree $0$. Hence we have $A(\xi)\varphi = \phi$ for any multiple of the primitive vector. The result then follows from Lemma~\ref{LogDiscrepancy}.
\end{proof}

\begin{prop}\label{chap5-prop-L-cts}
For a fixed norm $\chi$, the map 
$\xi\mapsto L^{\na}(\fs(\xi, \chi))$
is continuous.
\end{prop}

\begin{proof}
It is clear from the definition that $L^{\na}$ is upper-semicontinuous in $\xi$, since $\fs(\xi, \chi)$ is continuous in $\xi$. It suffices to show the other direction. Fix $\varepsilon>0$. By homogenity of $A(\xi)\varphi$, we only need to prove the result for normalized $\xi$.
For a sequence $\xi_j\to \xi$ of normalized Reeb fields, by~\cite[Lemma 2.11]{NAPP} we can find $j\gg 0$ such that $|\varphi_j -\varphi|<\frac{\varepsilon}{n}$ on $X_0$, where $\varphi_j = \fs(\xi_j, \chi)$, and $\varphi = \fs(\xi, \chi)$. Fix such a $j$, and take a sequence $v_m$ such that $\lim_{m\to\infty}(A(v_m)+n\varphi_j(v_m)) = L^{\na}(\varphi_j)$.
Now 
\[\lim_{m\to \infty} (A(v_m)+n\varphi_j(v_m))\geq \liminf_{m\to \infty} (A(v_m)+n\varphi(v_m) - \varepsilon)\geq L^{\na}(\varphi) - \varepsilon.\]
This shows the reverse inequality for any $\varepsilon>0$. The proof is complete letting $\varepsilon\to 0$.
\end{proof}
As a consequence, we have the remaining part of Theorem~\ref{thm-main-cts}.
\begin{cor}\label{chap5-cor-ding-cts}
For a fixed norm $\chi$, $D^{\na}(\fs(\xi,\chi))$ is continuous in $\xi$.
\end{cor}
\begin{proof}
    This is a consequence of Theorem~\ref{chap5-thm-mabuchi-cts} and Proposition~\ref{chap5-prop-L-cts}.
\end{proof}

\begin{proof}[Proof of Theorem \ref{thm-main-cts}]
The theorem follows from the combination of Proposition~\ref{chap5-prop-local-global-mabuchi}, Theorem~\ref{chap5-thm-mabuchi-cts}, Proposition~\ref{chap5-prop-local-global-ding} and Corollary~\ref{chap5-cor-ding-cts}.
\end{proof}
Moreover, the continuity results further yield a comparison between NA functionals and the Futaki invariant defined in~\cite{CSIrregular}.
\begin{cor}\label{chap5-cor-mabuchi-ding-comp}
    When $\varphi$ comes from a normal test configuration $(\mathcal{X}, \mathcal{D}, \xi,\eta)$, that is $\varphi = \fs(\xi, \chi_\eta)$, where $\chi_\eta$ denotes the norm induced by the test configuration, we have 
    \[D^{\na}(\varphi)\leq M^{\na}(\varphi)\leq \mathrm{Fut}(\mathcal{X}, \mathcal{D}, \xi,\eta),\]
    with equality on weakly special test configurations.
\end{cor}
\begin{proof}
    If $\xi$ is rational, then it is a consequence of~\cite[Proposition 7.32]{BHJ17}, Proposition~\ref{chap5-prop-local-global-mabuchi} and Proposition~\ref{chap5-prop-local-global-ding}. In general, it follows from continuity results Theorem~\ref{chap5-thm-mabuchi-cts} and Proposition~\ref{chap5-prop-L-cts}.
\end{proof}

\begin{cor}\label{Dingstab}
Assume $\varphi =\fs(\xi, \chi_\eta)$ comes from a $\QQ$-Gorenstein test configuration $(\cX, \cB; \xi, \eta)$ with $A(\eta)=0$. Then the above function $\xi\mapsto L^{\na}(\fs(\xi, \chi_\eta))$ is constant. 
\end{cor}
\begin{proof}
    If $\xi$ is rational, then 
    \[L^{\na}(\fs(\xi, \eta)) = \mathrm{lct}(\mathcal{X}, \mathcal{B}; \mathcal{X}_0)-1\]
    by~\cite[Proposition A.14]{LWX}. The result then follows from Proposition~\ref{chap5-prop-L-cts} since the right hand side does not depend on $\xi$.
\end{proof}
\begin{remark}
    We remark that this corollary is not surprising. In fact, from the point of view of weighted cscK metrics, Sasaki-Einstein metrics should correspond to a weighted cscK metric on a fixed polarized pair, with weights depending on the Reeb fields, as explained in~\cite[Section 3.5]{Lahdili} and~\cite[Section 2.4]{Li21}. On the other hand, the works~\cite{Hanli, BLXZ, Li21} on the YTD conjecture for weighted K\"ahler-Ricci solitons suggest that the non-Archimedean $L$ functional does not depend on the weights, and hence should be independent of the Reeb fields. 
\end{remark}
Based on the above remark, it is natural to ask the following question:
\begin{question}
    For a fixed norm $\chi$, is $\xi\mapsto L^{\na}(\fs(\xi, \chi))$ a constant function on the Reeb cone?
\end{question}
As we now explain, the answer to this question is yes when $\chi$ comes from either of the following:
\begin{itemize}
    \item a normal test configuration $(\cX, \cB; \xi, \eta)$ with $\eta$ in the Reeb cone;
    \item a $\QQ$-Gorenstein test configuration.
\end{itemize}
Recall that by Theorem~\ref{chap3-thm-local-tc-char}, a normal test configuration $(\cX, \cB; \xi, \eta)$ with $\eta$ in the Reeb cone corresponds to a Weil divisor $E=\sum_i a_i E_i\subset Y\xrightarrow[]{\pi} X$ with integral coefficients, and $-E$ ample.
\begin{thm}
    Notation as above. If $(\cX, \cB+\cX_{0, \rm red})$ is sub-lc, then we have
    \[L^{\na}(\cX, \cB; \xi, \eta) = \min_i \{a_i^{-1}A_{(X, B)}(\ord_{E_i})\}.\]
    In particular, let 
    $\{v_i\}\subset X_0^\TT$ be valuations such that $a_i^{-1}\ord_{E_i} =(s_i\xi)*v_i$ for some $s_i>0$. Then
    \[L^{\na}(\cX, \cB; \xi, \eta) = \min_i \{A_{(X, B)}(v_i)+A(\xi)\varphi(v_i)\}.\]
\end{thm}
\begin{proof}
    First assume $\xi$ rational and primitive. Denote by $(V, B_V)$ the corresponding quotient log Fano pair, and $(\cV, \cB; \cL)$ the test configuration for $(V, B_V; L=-A(\xi)^{-1}(K_V+B_V))$ induced by the same filtration.
    We will write 
    \[R= \bigoplus_m R_m\coloneqq \bigoplus_m\left(\bigoplus_{\la \xi, \alpha\ra=m} R_\alpha\right).\]
    Write $\cV_0=\sum b_i F_i$. Put $\cL = \pi^*L_{\AA^1}+D$, where $p: \cV\to V\times \AA^1$. We first assume $\ord_{F_i}(D)>0$ for any $i$.
    By~\cite[Lemma 5.17]{BHJ17}, we have
    \begin{align*}
        \cF^\lambda R &= \bigcap_i\bigoplus_m\{f_m\in R_m: b_i^{-1}(r_V(\ord_{F_i})+m\ord_{F_i}(D))(f_m)\geq \lambda)\}\\
        &= \bigcap_i \{f\in R: (t_i\xi)*r_V(\ord_{F_i})(f)\geq \lambda b_i\}
    \end{align*}
    where $r_V(\ord_{F_i})$ denotes the restriction of $\ord_{F_i}$ to $V$, and $t_i = \ord_{F_i}(D)>0$. On the other hand, by~\cite[Lemma 2.21]{LWX}, we see that 
    \[\cF^\lambda R = \bigcap_i\{f: a_i^{-1}\ord_{E_i}(f)\geq \lambda\}.\]
    Note that by construction, we have
    \begin{align*}
        &\cV_0 = \Proj \bigoplus_m\bigoplus_k \cF^kR_m/\cF^{>k}R_m,\\
        &E = \Proj \bigoplus_k\bigoplus_m \cF^kR_m/\cF^{>k}R_m,\\
        &\cX_0 = \Spec \bigoplus_k\bigoplus_m \cF^kR_m/\cF^{>k}R_m.
    \end{align*}
    Since each irreducible component of $\cX_0$ is $\TT\times \GG_m$ invariant, $\cV_0$ and $E$ have the same number of irreducible components. 
    By~\cite[Lemma 1.7]{BHJ17}, possibly after rearranging, we have for any $i$,
    \[a_i^{-1}\ord_{E_i} = b_i^{-1}(t_i\xi)*r_V(\ord_{F_j}).\]
    On the other hand, denote by $\varphi\in\cH^{\na}(\xi)$ the function induced by this filtration. Then by~\cite[Proposition 7.29]{BHJ17} one has
    \begin{align*}
        L^{\na}(\varphi) &= \min_i\{A_{(V, B_V)}(b_i^{-1}r_V(\ord_{F_i}))+A(\xi)\varphi(b_i^{-1}r_V(\ord_{F_i})\}\\
        &=\min_i\{A_{(V, B_V)}(b_i^{-1}r_V(\ord_{F_i}))+ A(\xi)b_i^{-1}\ord_{F_i}(D)\}\\
        &=\min_i\{b_i^{-1}A_{(X, B)}((t_i\xi)*r_V(\ord_{F_i})\}\\
        &=\min_i \{a_i^{-1} A_{(X, B)}(\ord_{E_i})\}.
    \end{align*}
    We now deal with the case when some $\ord_{F_i}(D)\leq 0$. In this case, the twisted filtration $\{\cF^{\bullet-pm}R_m\}$ for $p$ large enough satisfies the previous assumption. In terms of test configurations, this is nothing but twisting $\cL$ by $p\cV_0$, and so $L^{\na}(\varphi)$ is translated by $pA(\xi)$. On the other hand, $Y$ is replaced by a weighted blowup $Y'$ with respect to the filtration and $v_\xi=\ord_V$ with weights $(1, p)$. Pick $p$ large and divisible enough such that $a_i|p$ for any $i$. Then the exceptional divisor $E'\subset Y'$ is given by 
    \[E'=\Proj \bigoplus_d\left(\bigoplus_{pm+k=d}\cF^kR_m/\cF^{>k}R_m\right)\]
    with $E' = \sum_i a_i E_i'$ and $\ord_{E_i'} = (\frac{p}{a_i}\xi)*\ord_{E_i}$. We are now in the previous case. Let $\varphi_p$ be the function corresponding to the twisted filtration. Then
    \[L^{\na}(\varphi) = L^{\na}(\varphi_p)-pA(\xi) = \min_i\{a_i^{-1}A_{(X, B)}(\ord_{E_i'})\} -pA(\xi) = \min_i\{a_i^{-1}A_{(X, B)}(\ord_{E_i})\}.\]
    Note that the right hand side does not depend on $\xi$, and that $L^{\na}$ is homogeneous in $\xi$ of degree $0$, we have shown the equality for any $\xi$ rational. In general, we conclude by Proposition~\ref{chap5-prop-L-cts}.

    To see the second assertion, note that we already have \[A_{(X, B)}(a_i^{-1}\ord_{E_i})\geq A_{(X, B)}(a_i^{-1}\ord_{E_i})+A(\xi)\varphi(a_i^{-1}\ord_{E_i})\geq L^{\na}(\varphi)\] 
    since $\varphi(a_i^{-1}\ord_{E_i})\leq 0$. Thus, after taking minimum, the inequalities above become equalities.
    Let $v_i\in X_0$ be valuations such that $a_i^{-1}\ord_{E_i} = s_i\xi*v_i$ for some $s_i>0$. Then 
    \[L^{\na}(\varphi) = \min_i \{A_{(X, B)}(a_i^{-1}\ord_{E_i})+A(\xi)\varphi(a_i^{-1}\ord_{E_i})\} =  \min_i \{A_{(X, B)}(v_i)+A(\xi)\varphi(v_i)\}.\]
\end{proof}

\begin{cor}
    If $\varphi$ comes from a normal test configuration $(\cX,\cB; \xi, \eta)$ with $\eta$ in the Reeb field, then $L^{\na}(\varphi)$ is constant on the Reeb cone.
\end{cor}
\begin{proof}
    The above theorem in particular shows that for any $\xi$ rational, if we denote by $(V, B_V)$ the log Fano quotient by $\la \xi\ra$, and $(\cV_m, \cB_m)$ the normalized blowup of $V\times \AA^1$ along the corresponding flag ideal $\cI_m^{(\lambda_{\rm max})}$ induced by the same filtration on $V$, then $L^{\na}(\varphi)=\lim_m L^{\na}(\cV_m, \cB_m; \eta)$, as shown in e.g.~\cite[Lemma 4.3]{FujitaVolume},~\cite[Lemma 4.7]{FujitaVal}. Since this is true for any rational $\xi$, it follows that $L^{\na}(\varphi)$ is constant on the Reeb cone.
\end{proof}

\begin{cor}
    For $\varphi = \fs(\xi, \chi_\eta)$ coming from a fixed $\QQ$-Gorenstein test configuration $(\cX, \cB; \xi, \eta)$, $L^{\na}(\varphi)$ is constant on the Reeb cone. 
\end{cor}
\begin{proof}
   By the previous theorem, and~\cite[Section 2.5, Equation (109) and (116)]{Li21}, it suffices to deal with the case when $A(\eta)<0$. Fix $\xi'$ rational.
   First note that if $A(\eta)=0$ and $\lambda =\frac{A(\lambda)}{A(\xi')}\xi'+\lambda' \in \ft_\QQ^+$, then 
   \[L^{\na}(\fs(\xi, \eta+\lambda)) = L^{\na}(\fs(\xi, \eta+\lambda'))+A(\lambda) = L^{\na}(\fs(\xi', \eta))+A(\lambda) = L^{\na}(\fs(\xi, \eta))+A(\lambda),\]
   where the second to last equality follows from~\cite[Proposition A.14]{LWX} and the fact that twisting by $\lambda$ does not change the test configuration.
   Now assume $A(\eta)<0$ and pick $\lambda\in \ft_\QQ^+$ such that $A(\eta+\lambda)>0$. Put
   \[\lambda = \frac{A(\eta+\lambda)}{A(\xi')}\xi'+\lambda'\eqqcolon c\xi'+\lambda'.\]
   Note that $A(\eta+c\xi')=0$ and $A(\lambda') = -A(\eta)>0$.
   Then by the previous observation, we have
   \begin{align*}
   L^{\na}(\fs(\xi, \eta+\lambda)) &= L^{\na}(\fs(\xi', \eta+c\xi'))+A(\lambda') = L^{\na}(\fs(\xi', \eta+c\xi'))-A(\eta)\\
   &= L^{\na}(\fs(\xi', \eta))+A(\eta+\lambda)-A(\eta)\\
   &= L^{\na}(\fs(\xi', \eta))+A(\lambda).
   \end{align*}
   Again this calculation works for any fixed rational $\xi'$ and by the previous theorem, the left hand side doesn't depend on $\xi$. Hence $L^{\na}(\fs(\xi', \eta))$ doesn't depend on $\xi$ either.
\end{proof}
\section{Non-Archimedean characterization of K-semistability}\label{section: pfmain}
We are now in the place of proving Theorem~\ref{main-thm-stability-equivalence}. Since the energy functionals in this paper are slightly different from the ones in~\cite{wu, NAPP}, we will temporarily denote by $\tilde{E}$ the Monge--Amp\`ere energy studied in~\cite{wu}. Then 
\[E(\varphi) = A(\xi)\tilde{E}(\varphi)\]
for all $\varphi\in \cH^{\na}(\xi)$.
\begin{thm}[= Theorem \ref{main-thm-stability-equivalence}]
    Let $(X, B; \xi)$ be a log Fano cone singularity. Then the following are equivalent:
    \begin{enumerate}
        \item $(X, B; \xi)$ is K-semistable;
        \item $M^{\na}\geq 0$ on $\cH^{\na}(\xi)$;
        \item $D^{\na}\geq 0$ on $\cH^{\na}(\xi)$;
        \item $(X, B; \xi)$ is K-semistable with respect to special test configurations.
        \item $\delta(X, B;\xi)\geq 1$
    \end{enumerate}
\end{thm}
\begin{proof}
     By Corollary~\ref{chap5-cor-mabuchi-ding-comp}, we have~(3) $\Rightarrow$~(2) $\Rightarrow$~(1). By~\cite[Theorem 1.0.3]{HuangThesis} we have~(4) $\Leftrightarrow$~(5). We are going to show~(5) $\Leftarrow$~(3). This will complete the proof of the theorem since~(1) $\Rightarrow$~(4) trivially.

     To this end, we follow the proof of~\cite[Theorem 7.3]{BBJ}. To simplify the notation, write $\delta \coloneqq\delta(X, B;\xi)\geq 1$. In our notation (see Section~\ref{section: localval}) this reads
     \[A(v)\geq \delta \frac{n+1}{n}A(\xi)S(v;\xi),\]
     for all quasi-monomial $v\in X_0^\TT$. Note also that for any $\varphi\in \cH^{\na}(\xi)$, the convex combination $\delta^{-1}\varphi+(1-\delta^{-1})\varphi_\xi\in \cpsh( \xi)$ by~\cite[Proposition 3.3]{NAPP}. Thus by concavity of the Monge--Amp\`ere energy, we have
     \[
     \delta^{-1} \tilde{E}^{\na}(\varphi)=\delta^{-1} \tilde{E}^{\na}(\varphi)+(1-\delta^{-1})\tilde{E}^{\na}(\varphi_\xi)\leq \tilde{E}^{\na}(\delta^{-1}\varphi+(1-\delta^{-1})\varphi_\xi).
     \]
     Fix $\varphi\in\cH^{\na}(\xi)$ with $\sup_{X_0}\varphi\leq 0$. So~\cite[Proposition 7.3 and Proposition 7.4]{NAPP} yield:
     
      \begin{align*}
         \delta^{-1}E^{\na}(\varphi)&\leq E^{\na}(\delta^{-1}\varphi+(1-\delta^{-1})\varphi_\xi) = A(\xi)\tilde{E}^{\na}(\delta^{-1}\varphi+(1-\delta^{-1})\varphi_\xi) \\
         &= A(\xi)\inf_{\mu\in \cM^1}\left(E^\vee(\mu;\xi)+\int(\delta^{-1}\varphi+(1-\delta^{-1})\varphi_\xi)\mu\right)\\
         &\leq A(\xi)\inf_{v\in X_0^{\TT, \mathrm{\qm}}}\left(E^\vee(v;\xi)+\delta^{-1}\varphi(v)+(1-\delta^{-1})\varphi_\xi(v)\right)\\
         &\leq A(\xi)\inf_{v\in X_0^{\TT, \mathrm{\qm}}}\left(\frac{n+1}{n}S(v;\xi)+\delta^{-1}\varphi(v)+(1-\delta^{-1})\varphi_\xi(v)\right)\\
         &\leq  \inf_{v\in X_0^{\TT, \mathrm{\qm}}}\left(\delta^{-1}A(v)+ A(\xi)\delta^{-1}\varphi(v)+A(\xi)(1-\delta^{-1})\varphi_\xi(v)\right)\\
         &= \delta^{-1}\inf_{v\in X_0^{\TT, \mathrm{\qm}}}\left(A(v)+A(\xi)\varphi(v)\right)=\delta^{-1}L^{\na}(\varphi),
     \end{align*}
     where we have used $\varphi_\xi= 0$ on $X_0$ in the last equality.
     Hence 
     \[D^{\na}(\varphi) = L^{\na}(\varphi)-E^{\na}(\varphi)\geq 0.\]
\end{proof}

\bibliographystyle{alpha}
\bibliography{bib}

\end{document}